\documentclass[12pt,reqno]{amsart}

\usepackage{geometry}
\usepackage{hhline}

\usepackage{mathtools}

\usepackage{mathdots}

\usepackage{color}

\usepackage{pdfsync}

\usepackage{enumitem}

\usepackage{wasysym}

\usepackage{amssymb}

\usepackage{mathrsfs}

\DeclareMathAlphabet{\mathpzc}{OT1}{pzc}{m}{it}

\usepackage[all]{xy}

\usepackage{tikz}
\usetikzlibrary{arrows,matrix,decorations.pathmorphing,decorations.pathreplacing,positioning,shapes.geometric,shapes.misc,decorations.markings,decorations.fractals,calc,patterns}

\usepackage{graphicx}

\usepackage{float}

\usepackage[bottom]{footmisc}

\usepackage{moreenum}

\usepackage{makecell}

\entrymodifiers={+!!<0pt,\fontdimen22\textfont2>}

\newcommand{\dt}{\dot}
\newcommand{\ddt}{\ddot}

\usepackage{scalerel,stackengine}
\stackMath
\newcommand\newcheck[1]{%
\savestack{\tmpbox}{\stretchto{%
  \scaleto{%
    \scalerel*[\widthof{\ensuremath{#1}}]{\kern-.6pt\bigwedge\kern-.6pt}%
    {\rule[-\textheight/2]{1ex}{\textheight}}
  }{\textheight}%
}{0.5ex}}%
\stackon[1pt]{#1}{\scalebox{-1}{\tmpbox}}%
}
\stackMath
\newcommand\newhat[1]{%
\savestack{\tmpbox}{\stretchto{%
  \scaleto{%
    \scalerel*[\widthof{\ensuremath{#1}}]{\kern-.6pt\bigwedge\kern-.6pt}%
    {\rule[-\textheight/2]{1ex}{\textheight}}
  }{\textheight}%
}{0.5ex}}%
\stackon[1pt]{#1}{\scalebox{1}{\tmpbox}}%
}

\def\cA{\mathscr{A}}
\def\cB{\mathscr{B}}
\def\cC{\mathscr{C}}
\def\cD{\mathscr{D}}
\def\cE{\mathscr{E}}
\def\cF{\mathscr{F}}
\def\cG{\mathscr{G}}

\def\cK{\mathscr{K}}

\def\cM{\mathscr{M}}

\def\cS{\mathscr{S}}
\def\cT{\mathscr{T}}
\def\cU{\mathscr{U}}
\def\cV{\mathscr{V}}

\def\cX{\mathscr{X}}
\def\cY{\mathscr{Y}}

\def\add{\operatorname{add}}

\def\adots{\mathinner{\mkern1mu\raise1.0pt\vbox{\kern7.0pt\hbox{.}}\mkern2mu\raise5.0pt\hbox{.}\mkern2mu\raise9.0pt\hbox{.}\mkern1mu}}

\def\b{\operatorname{b}}

\def\Coker{\operatorname{Coker}}

\newcommand\ConditionEShort[3]
{\mbox{$#2( #3,\Sigma^{ #1 }#3 ) = 0$}}
\def\cone{\operatorname{cone}}

\def\dddots{\mathinner{\mkern1mu\raise10.0pt\vbox{\kern7.0pt\hbox{.}}\mkern2mu\raise5.3pt\hbox{.}\mkern2mu\raise1.0pt\hbox{.}\mkern1mu}}
\def\dddotssmall{\mathinner{\mkern1mu\raise7.0pt\vbox{\kern7.0pt\hbox{.}}\mkern-1mu\raise4pt\hbox{.}\mkern-1mu\raise1.0pt\hbox{.}\mkern1mu}}

\def\dim{\operatorname{dim}}

\def\dual{\operatorname{D}}
\def\End{\operatorname{End}}

\def\Ext{\operatorname{Ext}}

\def\Hom{\operatorname{Hom}}
\def\id{\operatorname{id}}
\def\Image{\operatorname{Im}}

\def\K{\operatorname{K}}
\def\K0{\operatorname{K}_0}

\def\Ker{\operatorname{Ker}}

\def\length{\operatorname{length}}

\def\mod{\operatorname{mod}}

\def\opp{\operatorname{op}}

\def\prj{\operatorname{prj}}

\def\PSL2{\operatorname{PSL}_2}

\def\SL2{\operatorname{SL}_2}

\numberwithin{equation}{section}

\newtheorem{Lemma}{Lemma}[section]
\newtheorem{Theorem}[Lemma]{Theorem}
\newtheorem{Proposition}[Lemma]{Proposition}

\theoremstyle{definition}
\newtheorem{Definition}[Lemma]{Definition}
\newtheorem{Setup}[Lemma]{Setup}

\newtheorem{Remark}[Lemma]{Remark}

\newtheorem{Example}[Lemma]{Example}

\newtheorem{ThmIntro}{Theorem}

\newtheorem*{bfhpg*}{}

  {\begin{list}{}{%
    \settowidth{\labelwidth}{\textbf{#1:}}%
    \setlength{\leftmargin}{\labelwidth}\addtolength{\leftmargin}{\labelsep}}}%
  {\end{list}}

\makeatletter
\@namedef{subjclassname@2020}{%
  \textup{2020} Mathematics Subject Classification}
\makeatother

\begin{document}

\setlength{\parindent}{0pt}
\setlength{\parskip}{7pt}

\title[Abelian subcategories induced by simple minded systems]{Abelian subcategories of triangulated categories induced by simple minded systems}

\author{Peter J\o rgensen}

\address{Department of Mathematics, Aarhus University, Ny Munkegade 118, 8000 Aarhus C, Denmark}
\email{peter.jorgensen@math.au.dk}

\urladdr{https://sites.google.com/view/peterjorgensen}




\begin{abstract} 

If $k$ is a field, $A$ a finite dimensional $k$-algebra, then the simple $A$-modules form a {\em simple minded collection} in the derived category $\cD^{ \b }( \mod\,A )$.  Their extension closure is $\mod\,A$; in particular, it is abelian.  This situation is emulated by a general simple minded collection $\cS$ in a suitable triangulated category $\cC$.  In particular, the extension closure $\langle \cS \rangle$ is abelian, and there is a tilting theory for such abelian subcategories of $\cC$.  These statements follow from $\langle \cS \rangle$ being the heart of a bounded $t$-structure.

\medskip
\noindent
It is a defining characteristic of simple minded collections that their negative self extensions vanish in every degree.  Relaxing this to vanishing in degrees $\{ -w+1, \ldots, -1 \}$ where $w$ is a positive integer leads to the rich, parallel notion of {\em $w$-simple minded systems}, which have recently been the subject of vigorous interest.  

\medskip
\noindent
If $\cS$ is a $w$-simple minded system for some $w \geqslant 2$, then $\langle \cS \rangle$ is typically not the heart of a $t$-structure.  Nevertheless, using different methods, we will prove that $\langle \cS \rangle$ is abelian and that there is a tilting theory for such abelian subcategories.  Our theory is based on Quillen's notion of exact categories, in particular a theorem by Dyer which provides exact subcategories of triangulated categories.

\medskip
\noindent
The theory of simple minded systems can be viewed as ``negative cluster tilting theory''.  In particular, the result that $\langle \cS \rangle$ is an abelian subcategory is a negative counterpart to the result from (higher) positive cluster tilting theory that if $\cT$ is a cluster tilting subcategory, then $( \cT * \Sigma \cT )/[ \cT ]$ is an abelian quotient category.

\end{abstract}

\maketitle

\setcounter{section}{-1}
\section{Introduction}
\label{sec:introduction}

Let $k$ be a field, $A$ a finite dimensional $k$-algebra, and let $\cS$ be a {\em simple minded collection} in the bounded derived category of finitely generated $A$-modules $\cD^{ \b }( \mod\,A )$, see Definition \ref{def:SMS}.  We have $\Ext^{ <0 }( \cS,\cS ) = 0$ by definition, and a canonical example is that $\cS$ is the set of simple $A$-modules.  The extension closure $\langle \cS \rangle$ is abelian and there is a Happel--Reiten--Smal\o\ tilting theory for such abelian subcategories of $\cD^{ \b }( \mod\,A )$.  Indeed, these statements follow from $\langle \cS \rangle$ being the heart of a bounded $t$-structure, see
\cite[def.\ 3.2, prop.\ 5.4, and sec.\ 7.2]{KY}.

This paper will investigate {\em simple minded systems}, which were introduced in \cite[def.\ 2.1]{KL}, \cite[sec.\ 1.2]{CS}, and \cite[def.\ 2.1]{CSP1} as a rich, parallel notion to simple minded collections.  Let $w \geqslant 2$ be an integer, $Q$ a finite acyclic quiver, and let $\cS$ be a $w$-simple minded system in the negative cluster category $\cC_{ -w }( Q ) = \cD^{ \b }( \mod\,kQ ) / \tau\Sigma^{ w+1 }$, see Definition \ref{def:SMS}.  We have $\Ext^{ \ell }( \cS,\cS ) = 0$ for $\ell \in \{ -w+1, \ldots, -1 \}$ by definition, and a canonical example is that $\cS$ is the set of simple $kQ$-modules.  Among the main results of this paper are that the extension closure $\langle \cS \rangle$ is abelian and that there is a ``simple minded'' tilting theory for such abelian  subcategories of $\cC_{ -w }(Q)$, see Theorems \ref{thm:A}, \ref{thm:B}, and \ref{thm:C}.  These statements cannot be proven using $t$-structures.  Indeed, $\cC_{ -w }( Q )$ has no $t$-structures with non-zero heart because it is $(-w)$-Calabi--Yau, see \cite[sec.\ 5.1]{HJY} and \cite[sec.\ 8.4]{KellerOrbit}.  Instead, our theory is based on Quillen's notion of exact categories, in particular a theorem by Dyer which provides exact subcategories of triangulated categories, see \cite[p.\ 1]{Dyer} and Section \ref{sec:exact}.

That $\langle \cS \rangle$ is an abelian subcategory is a negative cluster analogue of $( \cT * \Sigma \cT )/[ \cT ]$ being an abelian quotient category when $\cT$ is a cluster tilting subcategory of a positive cluster category.  Figure \ref{fig:table} elaborates on this analogy.  Items with the same position in different cells correspond to each other.  The column titled ``Cluster categories'' explains that simple minded systems constitute ``negative cluster tilting theory''.  This viewpoint has recently been explored vigorously, see \cite{CS2}, \cite{CS}, \cite{CSP1}, \cite{CSPP}, \cite{D2}, \cite{D}, \cite{IJ}, \cite{Jin}.  The row titled ``Simple minded theory'' sums up what we have said about the correspondence between simple minded collections and simple minded systems.  The categories $\cC_{ -w }( Q )$ were first studied in \cite{CS3}, and the notion of negative cluster categories was coined in \cite[sec.\ 1.2]{CSPP} based on $\cC_{ -w }( Q )$ being $(-w)$-Calabi--Yau.  

Note that we will actually develop our theory in a suitable general triangulated category $\cC$, and that the results do {\em not} extend to the case $w=1$, see Remark \ref{rmk:Thm_A_does_not_extend}.  
\begin{figure}
\begingroup
\begin{center}
\renewcommand{\arraystretch}{1.3}
\setlength{\tabcolsep}{2pt}
\begin{tabular}{c||l|l}
  & \multicolumn{1}{c|}{Derived categories and} & \multicolumn{1}{c}{Cluster}  \\
  & \multicolumn{1}{c|}{homotopy categories} & \multicolumn{1}{c}{categories}  \\
  \hhline{=#=|=} 
  \begin{tabular}{c}
    Simple minded \\ theory
  \end{tabular}
  &
  \begin{tabular}{l}
    Derived category $\cD^{\b}( \mod\,A )$ \\
    Simple minded collection $\cS$ \\
    $\Ext^{ <0 }( \cS,\cS ) = 0$ \\
    Abelian subcategory $\langle \cS \rangle$ \\
    Happel--Reiten--Smal\o\ tilting  of $\cS$ \\
    $t$-structure with heart $\langle \cS \rangle$
  \end{tabular} 
  &
  \begin{tabular}{l}
    Negative cluster category $\cC_{ -w }(Q)$ for $w \geqslant 2$ \\
    $w$-simple minded system $\cS$ \\
    $\Ext^{ -w+1,\ldots,-1 }( \cS,\cS ) = 0$ \\    
    {\color{red} Abelian subcategory $\langle \cS \rangle$} \\
    {\color{red} Simple minded tilting of $\cS$} \\
    ?
  \end{tabular} 
  \\
  \hline
  \begin{tabular}{c}
    Silting and \\ cluster tilting \\ theory
  \end{tabular}
  &
  \begin{tabular}{l}
    Homotopy category $\cK^{ \b }( \prj\,A )$ \\
    Silting subcategory $\cM$ \\
    $\Ext^{ >0 }( \cM,\cM ) = 0$ \\
    Abelian quotient $( \cM * \Sigma \cM )/[ \cM ]$ \\
    Silting mutation of $\cM$ \\
    Co-$t$-structure with coheart $\cM$
  \end{tabular}
  &
  \begin{tabular}{l}
    Positive cluster category $\cC_{ -w }(Q)$ for $w \leqslant -2$ \\
    Cluster tilting subcategory $\cT$ \\
    $\Ext^{ 1,\ldots,-w-1 }( \cT,\cT ) = 0$ \\
    Abelian quotient $( \cT * \Sigma \cT )/[ \cT ]$ \\
    Cluster mutation of $\cT$ \\
    ?
  \end{tabular}
\end{tabular}
\end{center}
\endgroup
\caption{Some main items of ``simple minded'' theory, cluster tilting theory, and silting theory for different categories.  Items with the same position in different cells correspond to each other.  The $(-w)$-Calabi--Yau category $\cC_{ -w }( Q )$ is a negative cluster category if $w \geqslant 1$ and a positive cluster category if $w \leqslant -2$.  Red items represent the main results of this paper.}
\label{fig:table}
\end{figure}
\begin{figure}
\begingroup
\[
  \begin{tikzpicture}[xscale=1.05,yscale=0.65]

    \draw[thick] (-10,3) to (-2,3);
    \draw[thick] (-10,1) to (-2,1);
    \draw[thick] (-9,1) to (-9,3);
    \draw[thick] (-7,1) to (-7,3);
    \draw[thick] (-5,1) to (-5,3);
    \draw[thick] (-3,1) to (-3,3);
    \draw[thick,dotted] (-8.4,1) to (-7.6,3);
    \draw[thick,dotted] (-6.4,1) to (-5.6,3);
    \draw[thick,dotted] (-4.4,1) to (-3.6,3);
    \draw[black,fill=black,opacity=0.25] (-5,1) -- (-5,3) -- (-7,3) -- (-7,1) -- cycle;
    \draw [thick,decorate,decoration={brace,amplitude=2mm}] (-6.98,3.4) -- (-5.02,3.4);
    \draw (-6,4.5) node {$\langle \cS \rangle$};
    \draw [thick,decorate,decoration={brace,amplitude=2mm}] (-4.98,3.4) -- (-3.02,3.4);
    \draw (-4,4.5) node {$\Sigma^{-1}\langle \cS \rangle$};
    \draw (-6.5,2) node {$\cT$};
    \draw (-5.5,2) node {$\cF$};
    
    \draw[thick] (-10,-3) to (-2,-3);
    \draw[thick] (-10,-1) to (-2,-1);
    \draw[thick,dotted] (-9,-1) to (-9,-3);
    \draw[thick,dotted] (-7,-1) to (-7,-3);
    \draw[thick,dotted] (-5,-1) to (-5,-3);
    \draw[thick,dotted] (-3,-1) to (-3,-3);
    \draw[thick] (-8.4,-3) to (-7.6,-1);
    \draw[thick] (-6.4,-3) to (-5.6,-1);
    \draw[thick] (-4.4,-3) to (-3.6,-1);
    \draw[black,fill=black,opacity=0.25] (-4.4,-3) -- (-3.6,-1) -- (-5.6,-1) -- (-6.4,-3) -- cycle;
    \draw [thick,decorate,decoration={brace,mirror,amplitude=2mm}] (-6.38,-3.4) -- (-4.42,-3.4);
    \draw (-5.4,-4.5) node {$\langle L_{ \cS' }( \cS ) \rangle$};
    \draw (-5.5,-1.99) node {$\cF$};
    \draw (-4.46,-1.96) node {$\Sigma^{-1}\!\cT$};
    
    \draw[thick] (-10+9,3) to (-2+9,3);
    \draw[thick] (-10+9,1) to (-2+9,1);
    \draw[thick] (-9+9,1) to (-9+9,3);
    \draw[thick] (-7+9,1) to (-7+9,3);
    \draw[thick] (-5+9,1) to (-5+9,3);
    \draw[thick] (-3+9,1) to (-3+9,3);
    \draw[thick,dotted] (-8.4+9,1) to (-7.6+9,3);
    \draw[thick,dotted] (-6.4+9,1) to (-5.6+9,3);
    \draw[thick,dotted] (-4.4+9,1) to (-3.6+9,3);
    \draw[black,fill=black,opacity=0.25] (-5+9,1) -- (-5+9,3) -- (-7+9,3) -- (-7+9,1) -- cycle;
    \draw [thick,decorate,decoration={brace,amplitude=2mm}] (-6.98+9,3.4) -- (-5.02+9,3.4);
    \draw (-6+9,4.5) node {$\langle \cS \rangle$};
    \draw [thick,decorate,decoration={brace,amplitude=2mm}] (-4.98+5,3.4) -- (-3.02+5,3.4);
    \draw (-4+5,4.5) node {$\Sigma\langle \cS \rangle$};
    \draw (-6.5+9,2) node {$\cU$};
    \draw (-5.5+9,2) node {$\cG$};

    \draw[thick] (-10+9,-3) to (-2+9,-3);
    \draw[thick] (-10+9,-1) to (-2+9,-1);
    \draw[thick,dotted] (-9+9,-1) to (-9+9,-3);
    \draw[thick,dotted] (-7+9,-1) to (-7+9,-3);
    \draw[thick,dotted] (-5+9,-1) to (-5+9,-3);
    \draw[thick,dotted] (-3+9,-1) to (-3+9,-3);
    \draw[thick] (-8.4+9,-3) to (-7.6+9,-1);
    \draw[thick] (-6.4+9,-3) to (-5.6+9,-1);
    \draw[thick] (-4.4+9,-3) to (-3.6+9,-1);
    \draw[black,fill=black,opacity=0.25] (-4.4+7,-3) -- (-3.6+7,-1) -- (-5.6+7,-1) -- (-6.4+7,-3) -- cycle;
    \draw [thick,decorate,decoration={brace,mirror,amplitude=2mm}] (-6.38+7,-3.4) -- (-4.42+7,-3.4);
    \draw (-5.4+7,-4.5) node {$\langle R_{ \cS' }( \cS ) \rangle$};
    \draw (-6.5+9,-2) node {$\cU$};
    \draw (-5.45+7,-2) node {$\Sigma\cG$};

  \end{tikzpicture} 
\]
\endgroup
\caption{The left and right parts of this figure illustrate left and right tilting.  In each case there is a proper abelian subcategory of $\cC$, shown in grey, with a torsion pair as indicated.  Similar schematics have appeared in the literature, see e.g.\ \cite[p.\ 417]{R} or \cite[fig.\ 1]{W} which is the immediate inspiration for our figure.}
\label{fig:tilting}
\end{figure}
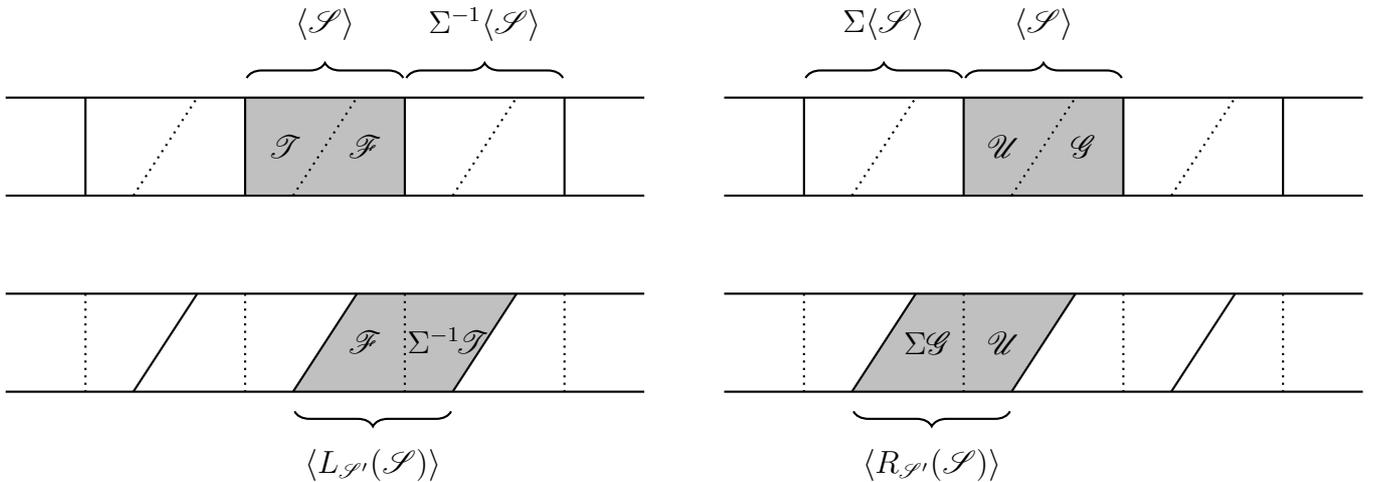

\smallskip
\subsection{Setup and a key definition}~\\
\label{subsec:setup}
\vspace{-2.4ex}

Before formulating the main results of the paper, we give a blanket setup and a key definition, which includes the definition of $w$-simple minded systems.

\begin{Setup}
\label{set:blanket}
Throughout the paper, $k$ is a field, $\cC$ an essentially small $k$-linear $\Hom$-finite triangulated category with split idempotents.  The suspension functor is denoted by $\Sigma$.
\end{Setup}

\begin{Definition}
[$w$-orthogonal collections, simple minded collections, and $w$-simple minded systems]
\label{def:SMS}
Let $w$ be a positive integer or $\infty$.
\begin{itemize}
\setlength\itemsep{4pt}

  \item  A {\em $w$-orthogonal collection} is a set $\cS = \{ s_1, \ldots, s_e \}$ of objects of $\cC$ satisfying the following two conditions, see \cite[sec.\ 1.2]{CS} and \cite[sec.\ 5]{Rickard}.
\medskip
\begin{enumerate}[leftmargin=40pt]
\setlength\itemsep{8pt}

  \item  $\cC( s_i,s_j )
= \left\{
    \begin{array}{cl}
      \mbox{a skew field} & \mbox{if $i=j$,} \\[1.5mm]
      0 & \mbox{if $i \neq j$}.
    \end{array}
  \right.$

  \item  $\cC( s_i,\Sigma^{ \ell }s_j ) = 0$ for any $i,j$ and $\ell \in \{ -w+1, \ldots, -1 \}$.  (This condition is void for $w=1$.)

\end{enumerate}

  \item  A {\em simple minded collection} is an $\infty$-orthogonal collection which also satisfies the following condition, see \cite[def.\ 3.2]{KY}.
\medskip
\begin{itemize}[leftmargin=40pt]
\setlength\itemsep{4pt}

  \item[(iii)]  Each object of $\cC$ is finitely built from $\cS$. 

\end{itemize}

  \item  A {\em $w$-simple minded system} for a positive integer $w$ is a $w$-orthogonal collection which also satisfies the following condition, see \cite[sec.\ 1.2]{CS} and \cite[def.\ 2.1]{CSP1}.  Here $\langle - \rangle$ denotes extension closure and the operation $*$ is explained in Definition \ref{def:star}.
\medskip
\begin{itemize}[leftmargin=40pt]
\setlength\itemsep{4pt}

  \item[(iii)']  $\cC = \langle \cS \rangle * \Sigma^{ -1 } \langle \cS \rangle * \cdots * \Sigma^{ -w+1 }\langle \cS \rangle$.
  
\end{itemize}

\end{itemize}
\end{Definition}

\begin{Remark}
A $w$-orthogonal collection is $w'$-orthogonal for each $w' < w$.  
\end{Remark}

\smallskip
\subsection{Abelian subcategories}~\\
\label{subsec:abelian}
\vspace{-2.4ex}

If $\cS$ is an $\infty$-orthogonal collection, then it is well known that $\langle \cS \rangle$ is abelian, see \cite[cor.\ 3 and prop.\ 4]{A}, \cite[prop.\ 5.4]{KY}, \cite[thm.\ 4.4]{Sch}, and \cite[rem.\ 1.3.14]{BBD}.  In this paper, $\cS$ is only a $w$-orthogonal collection for a positive integer $w$, but we will show that, nevertheless, $\langle \cS \rangle$ is abelian for $w \geqslant 2$.  The following is the first main result of the paper.  It follows from Theorem \ref{thm:16c} and Remark \ref{rmk:AB}.

\begin{ThmIntro}
\label{thm:A}
Let $w \geqslant 2$ be an integer, $\cS$ a $w$-orthogonal collection in $\cC$.  Then the extension closure $\langle \cS \rangle \subseteq \cC$ is a proper abelian subcategory where each object has finite length, and the simple objects are those in $\cS$ up to isomorphism.
\end{ThmIntro}

Note that the theorem includes the case of $\cS$ being a $w$-simple minded system for $w \geqslant 2$.  A {\em proper abelian subcategory of $\cC$} is an abelian subcategory whose short exact sequences are precisely those given by triangles in $\cC$, see Definition \ref{def:short_triangles_and_good_abelian_subcategories}(ii).  The classic example of a proper abelian subcategory is the heart of a $t$-structure, see \cite[thm.\ 1.3.6]{BBD}, but as explained at the start of the introduction, in this paper, $\langle \cS \rangle$ is not the heart of a $t$-structure.  Note that Theorem \ref{thm:A} does not extend to $1$-orthogonal collections, see Remark \ref{rmk:Thm_A_does_not_extend}.  

The following is the second main result of the paper.  It shows that if $\cS$ is a $w$-simple minded system, then we can say considerably more about $\langle \cS \rangle$.  It follows from Theorem \ref{thm:20} and Remark \ref{rmk:AB}.

\begin{ThmIntro}
\label{thm:B}
Let $w \geqslant 2$ be an integer, $\cS$ a $w$-simple minded system in $\cC$.  Then $\langle \cS \rangle \simeq \mod\,A$ for a finite dimensional $k$-algebra $A$.  If $k$ is algebraically closed, then the Gabriel quiver of $A$ has vertices the elements of $\cS = \{ s_1, \ldots, s_e \}$ and $\dim_k \cC( s_j,\Sigma s_i )$ arrows from $s_i$ to $s_j$. 
\end{ThmIntro}

\smallskip
\subsection{Simple minded tilting}~\\
\label{subsec:tilting}
\vspace{-2.4ex}

Stating the third main result of the paper requires some preparation.

If $w \geqslant 2$ is an integer, $\cS$ a $w$-simple minded system, $s' \in \cS$ an object, then $s'$ is simple in the abelian category $\langle \cS \rangle$ by Theorem \ref{thm:A}.  The extension closure $\langle s' \rangle$ is the subcategory of $\langle \cS \rangle$ consisting of objects with all composition factors equal to $s'$.  It is a minimal non-zero torsion class and $\langle \cS \rangle$ has a torsion pair $( \cT,\cF )$ with $\cT = \langle s' \rangle$.  

If $\langle \cS \rangle$ were the heart of a $t$-structure, then \cite[prop.\ 2.1]{HRS} would provide a left tilted proper abelian subcategory which had the torsion pair $( \cF,\Sigma^{-1}\cT )$.  Right tilting would be a mirror version.

In this paper, $\langle \cS \rangle$ is not the heart of a $t$-structure, but we will show that, nevertheless, if $\cC$ is $(-w)$-Calabi--Yau then tilting still works the same way.  To wit, there is a left tilted $w$-simple minded system $L_{ s' }( \cS )$ such that the proper abelian subcategory $\langle L_{ s' }( \cS ) \rangle$ has the torsion pair $( \cF,\Sigma^{-1}\cT )$, see Figure \ref{fig:tilting}.
Right tilting is a mirror version.

Tilted $w$-simple minded systems are defined as follows, see \cite[sec.\ 4]{CS}, \cite[def.\ 3.1]{CSP1}, \cite[def.\ 6.1]{D2}, \cite[def.\ 7.5]{KY}, \cite[sec.\ 2.2]{W}.  Note that we permit tilting at any subset of $\cS$, which requires no more work than tilting at $s'$.

\begin{Definition}
[Simple minded tilting]
\label{def:tilting}
Let $\cS$ be a $w$-orthogonal collection in $\cC$ and let $\cS',\cS'' \subseteq \cS$ be subsets.
\begin{enumerate}
\setlength\itemsep{10pt}

  \item  We say that {\em the left tilt $L_{ \cS ' }( \cS )$ of $\cS$ at $\cS'$ is defined} if each $s \in \cS \setminus \cS'$ has a $( \Sigma^{ -1 } \langle \cS' \rangle )$-cover $\Sigma^{-1}t_s \xrightarrow{ \tau_s } s$, which we complete to the following triangle.
\begin{equation}
\label{equ:L_triangle}
  \Sigma^{ -1 }t_s \xrightarrow{ \tau_s } s \xrightarrow{} L_{ \cS' }( s ) \xrightarrow{} t_s
\end{equation}
In this case $L_{ \cS' }( \cS )$ is the set consisting of the objects
\[
  L_{ \cS' }( s ) =
  \left\{
    \begin{array}{cl}
      \mbox{the object defined by \eqref{equ:L_triangle}} & \mbox{if $s \in \cS \setminus \cS'$,} \\[1.5mm]
      \Sigma^{ -1 }s & \mbox{if $s \in \cS'$.} \\      
    \end{array}
  \right.
\]

  \item  We say that {\em the right tilt $R_{ \cS'' }( \cS )$ of $\cS$ at $\cS''$ is defined} if each $s \in \cS \setminus \cS''$ has a $( \Sigma \langle \cS'' \rangle )$-envelope $s \xrightarrow{ \gamma_s } \Sigma g_s$, which we complete to the following triangle.
\begin{equation}
\label{equ:R_triangle}
  g_s \xrightarrow{} R_{ \cS'' }( s ) \xrightarrow{} s \xrightarrow{ \gamma_s } \Sigma g_s
\end{equation}
In this case $R_{ \cS'' }( \cS )$ is the set consisting of the objects
\[
  R_{ \cS'' }( s ) =
  \left\{
    \begin{array}{cl}

      \mbox{the object defined by \eqref{equ:R_triangle}} & \mbox{if $s \in \cS \setminus \cS''$,} \\[1.5mm]
      \Sigma s & \mbox{if $s \in \cS''$.} \\      
    \end{array}
  \right.
\]

\end{enumerate}
\end{Definition}

Note that left and right tilting are inverse operations, see Proposition \ref{pro:left_tilting_v_right_tilting}, and that if $\cS$ is a $w$-simple minded system, then left and right tilts are defined because $( \Sigma^{ -1 } \langle \cS' \rangle )$-covers and $( \Sigma \langle \cS'' \rangle )$-envelopes exist by Lemma \ref{lem:functorial_finiteness}.

We will use the notation
\[
  \cX^{ \perp } = \{\, c \in \cC \mid \cC( \cX,c ) = 0 \,\}
  \;\;,\;\;
  {}^{ \perp }\cX = \{\, c \in \cC \mid \cC( c,\cX ) = 0 \,\}
\]
when $\cX$ is a class of objects of $\cC$.  The following is the third main result of the paper.  It follows from Theorems \ref{thm:37}, \ref{thm:37_prime}, and Remark \ref{rmk:C}.

\begin{ThmIntro}
\label{thm:C}
Let $w \geqslant 2$ be an integer, $\cS$ a $w$-simple minded system in $\cC$, and assume that $\cC$ is $(-w)$-Calabi--Yau.  Let $\cS',\cS'' \subseteq \cS$ be subsets.

The left tilt $L_{ \cS' }( \cS )$ and the classes $\cT = \langle \cS' \rangle$, $\cF = \langle \cS' \rangle^{ \perp } \cap \langle \cS \rangle$ have the following properties.
\begin{enumerate}
\setlength\itemsep{4pt}

  \item  $L_{ \cS' }( \cS )$ is a $w$-simple minded system.

  \item  The abelian category $\langle \cS \rangle$ has the torsion pair $( \cT,\cF )$.

  \item  The abelian category $\langle L_{ \cS' }( \cS ) \rangle$ has the torsion pair $( \cF,\Sigma^{-1}\cT )$.

\end{enumerate}

The right tilt $R_{ \cS'' }( \cS )$ and the classes $\cU = {}^{ \perp }\langle \cS'' \rangle  \cap \langle \cS \rangle$, $\cG = \langle \cS'' \rangle$ have the following properties.
\begin{enumerate}
\setlength\itemsep{4pt}

  \item  $R_{ \cS'' }( \cS )$ is a $w$-simple minded system.

  \item  The abelian category $\langle \cS \rangle$ has the torsion pair $( \cU,\cG )$.

  \item  The abelian category $\langle R_{ \cS'' }( \cS ) \rangle$ has the torsion pair $( \Sigma\cG,\cU )$.

\end{enumerate}
\end{ThmIntro}

\smallskip
\subsection{Structure of the paper}~\\
\label{subsec:layout}
\vspace{-2.4ex}

The body of the paper proves Theorems \ref{thm:16c}, \ref{thm:20}, \ref{thm:37}, and \ref{thm:37_prime}.  It is explained in Remarks \ref{rmk:AB} and \ref{rmk:C} how they imply Theorems \ref{thm:A} through \ref{thm:C}.  The theorems in the body of the paper are not restricted to $w$-simple minded systems but apply to $w$-orthogonal collections under certain assumptions.

Section \ref{sec:background} collects a number of definitions and a lemma.  Section \ref{sec:exact} concerns exact subcategories of triangulated categories.  Section \ref{sec:abelian} proves Theorems \ref{thm:16c} and \ref{thm:20}, which imply Theorems \ref{thm:A} and \ref{thm:B}.  Section \ref{sec:torsion} shows several lemmas on torsion pairs in the abelian categories $\langle \cS \rangle$.  Sections \ref{sec:left_tilting} and \ref{sec:right_tilting} prove Theorems \ref{thm:37} and \ref{thm:37_prime}, which imply Theorem \ref{thm:C}.  Section \ref{sec:Type_A} applies Theorems \ref{thm:A}, \ref{thm:B}, and \ref{thm:C} to negative cluster categories of Dynkin type $A$.

\section{Background}
\label{sec:background}

This section collects a number of definitions and a lemma.  Recall that throughout the paper, $k$ is a field, $\cC$ an essentially small $k$-linear $\Hom$-finite triangulated category with split idempotents.  The suspension functor is denoted by $\Sigma$.

\begin{Definition}
[Additive subcategories]
\label{def:additive_subcategory}
If $\cX$ is an additive category, then an {\em additive subcategory $\cA \subseteq \cX$} is a full subcategory closed under isomorphisms, direct sums, and direct summands in $\cX$. 
\end{Definition}

\begin{Definition}
[Short triangles and proper abelian subcategories]
\label{def:short_triangles_and_good_abelian_subcategories}
$\;$\vspace{-1ex}
\begin{enumerate}
\setlength\itemsep{4pt}

  \item  A diagram $c' \xrightarrow{ \gamma' } c \xrightarrow{ \gamma } c''$ in $\cC$ is a {\em short triangle} if there exists a triangle $c' \xrightarrow{ \gamma' } c \xrightarrow{ \gamma } c'' \xrightarrow{ \gamma'' } \Sigma c$. 

  \item  An additive subcategory $\cA \subseteq \cC$ is a {\em proper abelian subcategory} if $\cA$ is an abelian category whose short exact sequences are precisely the short triangles in $\cC$ with terms in $\cA$.

  \item  If $\cA \subseteq \cC$ is an additive subcategory, then we set
\[
  \cE_{ \cA } =
  \{\,
    \mbox{short triangles $a' \xrightarrow{} a \xrightarrow{} a''$ with $a',a,a'' \in \cA$}
  \,\}.
\]

\end{enumerate}
\end{Definition}

\begin{Remark}
{\em Admissible abelian subcategories} are a classic notion introduced in \cite[def.\ 1.2.5]{BBD}.  They are stronger than proper abelian subcategories, requiring in addition that all their negative self extensions vanish.

{\em Distinguished abelian subcategories} were recently introduced in \cite[def.\ 1.1]{L}.  They are weaker than proper abelian subcategories, requiring only that each short exact sequence is a short triangle.
\end{Remark}

\begin{Definition}
[The star product of subcategories]
\label{def:star}
$\;$\vspace{-1ex}
\begin{enumerate}
\setlength\itemsep{4pt}

  \item  If $\cA, \cB \subseteq \cC$ are full subcategories, then
\[
  \cA * \cB = 
  \{\,
    e \in \cC \mid \mbox{there is a short triangle $a \xrightarrow{} e \xrightarrow{} b$ with $a \in \cA$, $b \in \cB$}
  \,\}
\]
is a full subcategory closed under isomorphisms.  If $\cA$ and $\cB$ are closed under direct sums then so is $\cA * \cB$.

  \item  If $\cA, \cB \subseteq \cC$ are additive subcategories with $\cC( \cA,\cB ) = 0$, then $\cA * \cB \subseteq \cC$ is an additive subcategory 
by \cite[prop.\ 2.1(1)]{IY}.

\end{enumerate}
\end{Definition}

\begin{Definition}
[Subcategories closed under extensions]
\label{def:closed_under_extensions}
An additive subcategory $\cA \subseteq \cC$ is {\em closed under extensions} if $\cA * \cA \subseteq \cA$.
\end{Definition}

\begin{Definition}
[Extension closure, see {\cite[sec.\ 2]{D}}]
\label{def:extension_closure}
Let $\cS \subseteq \cC$ be a full subcategory, $n$ a non-negative integer.  Define recursively
\[
  ( \cS )_n =
  \left\{
    \begin{array}{cl}
      \{ 0 \} & \mbox{if $n=0$, } \\[1.5mm]
      ( \cS )_{ n-1 } * ( \cS \cup \{ 0 \} ) & \mbox{if $n \geqslant 1$. } \\
    \end{array}
  \right.
\]
Then
\[
  \langle \cS \rangle = \bigcup_{ n \geqslant 0 } ( \cS )_n
\]
is the smallest full subcategory of $\cC$ which is closed under isomorphisms, direct sums, and extensions and contains $\cS$; see \cite[lem.\ 2.3]{D}.

If $\cS$ is a $w$-orthogonal collection for an integer $w \geqslant 1$, then $\langle \cS \rangle$ is the smallest additive subcategory of $\cC$ which is closed under extensions and contains $\cS$; see \cite[lem.\ 2.7]{D}.  
\end{Definition}

\begin{Definition}
[{Homotopy cartesian squares, see \cite[def.\ 1.4.1]{N}}]
\label{def:homotopy_cartesian_squares}
A {\em homotopy cartesian square in $\cC$} is a 
commutative square in $\cC$,
\[
\vcenter{
  \xymatrix @+0.5pc {
    c \ar^{\gamma}[r] \ar_{\varphi}[d] & d \ar^{\varphi'}[d] \\
    c' \ar_{\gamma'}[r] & d'\lefteqn{,} \\
                    }
        }
\]
such that $c \xrightarrow{ \begin{psmallmatrix} \varphi \\ \gamma \end{psmallmatrix} } c' \oplus d \xrightarrow{ ( -\gamma',\varphi' ) } d'$ is a short triangle. 
\end{Definition}

\begin{Definition}
[Inflations, deflations, and conflations]
\label{def:inflations_deflations_conflations}
$\;$\vspace{-2.0ex}
\begin{enumerate}
\setlength\itemsep{4pt}

  \item  A {\em kernel-co\-ker\-nel pair} in an additive category $\cX$ is a diagram $x' \xrightarrow{ \xi' } x \xrightarrow{ \xi } x''$ such that $\xi'$ is a kernel of $\xi$ and $\xi$ is a cokernel of $\xi'$. 

  \item  If $\cE$ is a fixed class of kernel-cokernel pairs in $\cX$, then {\em conflations}, {\em inflations}, and {\em deflations in $\cX$} are defined as follows:  The conflations are the diagrams in $\cE$.  The inflations, respectively deflations, are the first, respectively second, morphisms in conflations.

\end{enumerate}
\end{Definition}

\begin{Definition}
[Exact categories, see {\cite[def.\ 2.1]{B}}, {\cite[\Paragraph 2]{Q}}]
\label{def:exact_categories}
$\;$\vspace{-1ex}
\begin{enumerate}
\setlength\itemsep{4pt}

  \item  An {\em exact category} is a pair $( \cX,\cE )$ where $\cX$ is an additive category, $\cE$ a fixed set of kernel-cokernel pairs in $\cX$ satisfying the axioms [E0], [E0${}^{\opp}$], [E1], [E1${}^{\opp}$], [E2], [E2${}^{\opp}$] of \cite[def.\ 2.1]{B}.  We use this version of the definition for practical reasons; it is equivalent to the original definition given in \cite[\Paragraph 2]{Q}, see \cite[sec.\ 2]{B}.

  \item  An exact category $( \cX,\cE )$ has a fixed class $\cE$ of kernel-cokernel pairs, so the terms conflation, inflation, and deflation in $\cX$ make sense.  

\end{enumerate}
\end{Definition}

\begin{Definition}
[Projective objects]
\label{def:enough_projective_objects}
Let $( \cX,\cE )$ be an exact category.\vspace{-1ex}
\begin{enumerate}
\setlength\itemsep{4pt}

  \item  A {\em projective object of $\cX$} is an object $p \in \cX$ such that each conflation $x' \xrightarrow{} x \xrightarrow{} x''$  in $\cX$ induces a short exact sequence $0 \xrightarrow{} \cX( p,x' ) \xrightarrow{} \cX( p,x ) \xrightarrow{} \cX( p,x'' ) \xrightarrow{} 0$.

  \item  We say that $\cX$ {\em has enough projective objects} if each object $x$ permits a conflation $k \xrightarrow{} p \xrightarrow{} x$ where $p$ is projective.

\end{enumerate}
\end{Definition}

We end the background section with a convenient lemma due to \cite{CSP1} and \cite{D}.

\begin{Lemma}
\label{lem:functorial_finiteness}
If $w \geqslant 1$ is an integer, $\cS$ a $w$-simple minded system in $\cC$, and $\cS' \subseteq \cS$ is a subset, then $\Sigma^i\langle \cS' \rangle$ is functorially finite in $\cC$ for each $i$. 
\end{Lemma}

\begin{proof}
By \cite[thm.\ 2.11]{CSP1} it follows from the proof of \cite[thm.\ 3.3]{D} that $\langle \cS' \rangle$ is functorially finite in $\langle \cS \rangle$, and $\langle \cS \rangle$ is functorially finite in $\cC$ by \cite[cor.\ 2.9]{CSP1}.  Hence $\langle \cS' \rangle$ is functorially finite in $\cC$ and since $\Sigma$ is an automorphism, so is $\Sigma^i\langle \cS' \rangle$.  Note that the methods of \cite{CSP1} apply despite our definition of $w$-simple minded system being slightly more general than theirs.
\end{proof}

\section{A theorem by Dyer on exact subcategories of triangulated categories}
\label{sec:exact}

The main item of this section is a theorem by Dyer on exact subcategories of triangulated categories, see \cite[p.\ 1]{Dyer} and Theorem \ref{thm:Dyer}.  We are grateful to Professor Dyer for permitting the proof of Theorem \ref{thm:Dyer} to appear in this paper.  This section also proves Proposition \ref{pro:9} and Lemma \ref{lem:5b}, which appear to be new.

\begin{Setup}
\label{set:exact}
In this section $\cA \subseteq \cC$ is a fixed additive subcategory which is closed under extensions and satisfies \ConditionEShort{-1}{\cC}{\cA}.  
\end{Setup}

\begin{Lemma}
\label{lem:3}
Each diagram $a' \xrightarrow{ \alpha' } a \xrightarrow{ \alpha } a''$ in $\cE_{ \cA }$ is a kernel-cokernel pair in $\cA$.
\end{Lemma}

\begin{proof}
The diagram in the lemma is a short triangle, so there is a triangle $a' \xrightarrow{ \alpha' } a \xrightarrow{ \alpha } a'' \xrightarrow{} \Sigma a'$.  For $x \in \cA$ it induces the exact sequences
\begin{align*}
    \cC( x,\Sigma^{ -1 }a'' )
  & \xrightarrow{}
    \cC( x,a' )
    \xrightarrow{}
    \cC( x,a )
    \xrightarrow{}
    \cC( x,a'' ), \\
    \cC( \Sigma a',x )
  & \xrightarrow{}
    \cC( a'',x )
    \xrightarrow{}
    \cC( a,x )
    \xrightarrow{}
    \cC( a',x ).
\end{align*}
The first term of each sequence is zero because \ConditionEShort{-1}{\cC}{\cA} so the lemma follows.
\end{proof}

\begin{Lemma}
\label{lem:2}
Let the following be a homotopy cartesian square in $\cC$ with $a,x,a',x' \in \cA$.
\begin{equation}
\label{equ:homotopy_cartesian_square}
\vcenter{
  \xymatrix @+0.5pc {
    a \ar^{\alpha}[r] \ar_{\varphi}[d] & x \ar^{\varphi'}[d] \\
    a' \ar_{\alpha'}[r] & x' \\
                    }
        }
\end{equation}
Then \eqref{equ:homotopy_cartesian_square} is a pullback square and a pushout square in $\cA$.   
\end{Lemma}

\begin{proof}
Since \eqref{equ:homotopy_cartesian_square} is a homotopy cartesian square, there is a triangle $a \xrightarrow{ \begin{psmallmatrix} \varphi \\ \alpha \end{psmallmatrix} } a' \oplus x \xrightarrow{ ( -\alpha',\varphi' ) } x' \xrightarrow{} \Sigma a$ by Definition \ref{def:homotopy_cartesian_squares} whence the short triangle $a \xrightarrow{ \begin{psmallmatrix} \varphi \\ \alpha \end{psmallmatrix} } a' \oplus x \xrightarrow{ ( -\alpha',\varphi' ) } x'$ is a kernel-cokernel pair in $\cA$ by Definition \ref{def:short_triangles_and_good_abelian_subcategories}(iii) and Lemma \ref{lem:3}.  It follows that \eqref{equ:homotopy_cartesian_square} is a pullback square and a pushout square in $\cA$.
\end{proof}

\begin{Remark}
\label{rmk:NPexa2.13}
If $( \cX,\cE )$ is an essentially small $k$-linear exact category, then there is a bifunctor
\[
  \Ext_{ ( \cX,\cE ) }( -,- ) :
  \cX^{ \opp } \times \cX \xrightarrow{}
  \mod\,k
\]
defined by letting $\Ext_{ ( \cX,\cE ) }( x'',x' )$ be the set of equivalence classes in the usual sense of conflations $x' \xrightarrow{} x \xrightarrow{} x''$.  Addition of equivalence classes is defined by Baer sum.  See \cite[exa.\ 2.13]{NP} for details.
\end{Remark}

\begin{Theorem}
[Dyer]
\label{thm:Dyer}
The category $\cA$ from Setup \ref{set:exact} has the following properties.
\begin{enumerate}
\setlength\itemsep{4pt}

  \item  $( \cA,\cE_{ \cA } )$ is an exact category.  
  
  \item  For $a',a'' \in \cA$ there is a natural isomorphism 
\[
  \Ext_{ ( \cA,\cE_{ \cA } ) }( a'',a' )
  \xrightarrow{ \Phi }
  \cC( a'',\Sigma a' ).
\]

  \item  Let $p \in \cA$ be given.  Then $p$ is a projective object of $\cA$ $\Leftrightarrow$ $\cC( p,\Sigma \cA ) = 0$.

\end{enumerate}
\end{Theorem}

\begin{proof}
(i):  We will verify that $( \cA,\cE_{ \cA } )$ satisfies \cite[def.\ 2.1]{B}.  Since $\cA \subseteq \cC$ is an additive subcategory, $\cA$ is itself an additive category.  The diagrams in $\cE_{ \cA }$ are kernel-cokernel pairs in $\cA$ by Lemma \ref{lem:3}; in particular, we can talk about conflations, inflations, and deflations in $\cA$.  

We now check the labelled axioms from \cite[def.\ 2.1]{B}.

[E0] and [E0${}^{\opp}$]: If $a \in \cA$ then there are short triangles $a \xrightarrow{ \id_a } a \xrightarrow{} 0$ and $0 \xrightarrow{} a \xrightarrow{ \id_a } a$.  They are conflations in $\cA$ so $\id_a$ is an inflation and a deflation.  

[E1]: Let $a^0 \xrightarrow{ \alpha^0 } a^1$ and $a^1 \xrightarrow{ \alpha^1 } a^2$ be inflations.  Then there are conflations in $\cA$, that is, short triangles $a^0 \xrightarrow{ \alpha^0 } a^1 \xrightarrow{} x^2$ and $a^1 \xrightarrow{ \alpha^1 } a^2 \xrightarrow{} x^3$ with $x^2, x^3 \in \cA$.  By the octahedral axiom they can be combined to the following commutative diagram, where each row and column is a short triangle.
\[
  \xymatrix @+0.5pc {
    a^0 \ar^{\alpha^0}[r] \ar@{=}[d] & a^1 \ar[r] \ar^{ \alpha^1 }[d] & x^2 \ar[d] \\
    a^0 \ar_{\alpha^1\alpha^0}[r] \ar[d] & a^2 \ar[r] \ar[d] & x \ar[d] \\
    0 \ar[r] & x^3 \ar@{=}[r] & x^3 \\
                    }
\]
Since $\cA$ is closed under extensions the third column shows $x \in \cA$, whence the second row is a conflation in $\cA$ showing that $\alpha^1\alpha^0$ is an inflation.  

[E1${}^{\opp}$]: This axiom is superfluous by \cite[app.\ A]{K}, see \cite[rmk.\ 2.4]{B}.  

[E2]: Let the diagram
\begin{equation}
\label{equ:partial_square}
\vcenter{
  \xymatrix @+0.5pc {
    a \ar^{\alpha}[r] \ar_{\varphi}[d] & x \\
    a' \\
                    }
        }
\end{equation}
be given in $\cA$ where $\alpha$ is an inflation.  This means that there is a conflation in $\cA$, that is, a short triangle $a \xrightarrow{ \alpha } x \xrightarrow{} y$ with $y \in \cA$, so there is a triangle $\Sigma^{ -1 }y \xrightarrow{ \delta } a \xrightarrow{ \alpha } x \xrightarrow{} y$.  We can hence construct the solid part of the following commutative diagram with triangles as rows.
\[
\vcenter{
  \xymatrix @+0.5pc {
    \Sigma^{-1}y \ar^{\delta}[r] \ar@{=}[d] & a \ar^{\alpha}[r] \ar_{\varphi}[d] & x \ar[r] \ar@{.>}^{\varphi'}[d] & y \ar@{=}[d] \\
    \Sigma^{-1}y \ar_{\varphi\delta}[r] & a' \ar_{\alpha'}[r] & x' \ar[r] & y \\
                    }
        }
\]
By \cite[lem.\ 1.4.3]{N} the morphism $\varphi'$ can be chosen to complete the diagram such that the middle square is homotopy cartesian in $\cC$.  

Since $a',y$ are in $\cA$ which is closed under extensions, the second row shows that $x'$ is in $\cA$ and that $\alpha'$ is an inflation.  Lemma \ref{lem:2} implies that the middle square is a pushout square in $\cA$.  The square completes \eqref{equ:partial_square}, so we have established [E2].

[E2${}^{\opp}$]: Proved similarly to [E2].

(ii):  To define $\Phi$, let $\varepsilon = [ a' \xrightarrow{ \alpha' } a \xrightarrow{ \alpha } a'' ]$ in $\Ext_{ ( \cA,\cE_{ \cA } ) }( a'',a' )$ be given.  The conflation $a' \xrightarrow{ \alpha' } a \xrightarrow{ \alpha } a''$ is a short triangle, and we let $a' \xrightarrow{ \alpha' } a \xrightarrow{ \alpha } a'' \xrightarrow{ \alpha'' } \Sigma a'$ be a triangle and set $\Phi( \varepsilon ) = \alpha''$.  It is straightforward to check that this gives a well defined, natural, $k$-linear bijection.  It is useful to note that when $a' \xrightarrow{ \alpha' } a \xrightarrow{ \alpha } a''$ is given, the morphism $\alpha''$ is unique by \cite[cor.\ 1.1.10(ii)]{BBD} because $\cC( \Sigma a',a'' ) = 0$ since \ConditionEShort{-1}{\cC}{\cA}.

(iii):  Since $\Ext_{ ( \cA,\cE_{ \cA } ) }$ is defined as in Remark \ref{rmk:NPexa2.13}, it follows from \cite[prop.\ 11.3(iii)]{B} that $p$ is a projective object of $\cA$ if and only if $\Ext_{ ( \cA,\cE_{ \cA } ) }( p,\cA ) = 0$.  By  part (ii) this is equivalent to $\cC( p,\Sigma \cA ) = 0$.  
\end{proof}

The idea of the following proposition is due to \cite[thm.\ 2.1]{CSPP}.

\begin{Proposition}
\label{pro:9}
Assume that each object of $\cA$ has a $\Sigma \cA$-envelope.  Then the exact category $( \cA,\cE_{ \cA } )$ has enough projective objects.
\end{Proposition}

\begin{proof}
Given $a \in \cA$, pick a $\Sigma \cA$-envelope $a \xrightarrow{} \Sigma a'$ and complete to a triangle  $a' \xrightarrow{} q \xrightarrow{} a \xrightarrow{} \Sigma a'$.  Then the short triangle $a' \xrightarrow{} q \xrightarrow{} a$ is a conflation in $\cA$ because $q \in \cA$ since $\cA$ is closed under extensions.  Moreover, the triangulated Wakamatsu Lemma, dual to \cite[lem.\ 2.1]{JorARSub}, says $\cC( q,\Sigma \cA ) = 0$, so $q$ is a projective object of $\cA$ by Theorem \ref{thm:Dyer}(iii).
\end{proof}

\begin{Lemma}
\label{lem:5b}
Assume that $\cA * ( \Sigma \cA ) \subseteq ( \Sigma \cA ) * \cA$.  Then each categorical monomorphism in $\cA$ is an inflation in $\cA$.  
\end{Lemma}

\begin{proof}
Let $a^0 \xrightarrow{ \alpha^0 } a^1$ be a categorical monomorphism in $\cA$ and complete to a triangle $\Sigma^{ -1 }c \xrightarrow{ \delta } a^0 \xrightarrow{ \alpha^0 } a^1 \xrightarrow{} c$.  We will show $c \in \cA$ whence the triangle gives a conflation $a^0 \xrightarrow{ \alpha^0 } a^1 \xrightarrow{} c$ in $\cA$ showing that $\alpha^0$ is an inflation.  

The triangle shows $c \in \cA * ( \Sigma \cA )$.  Hence $c \in ( \Sigma \cA ) * \cA$ so there exists a triangle $\Sigma a' \xrightarrow{ \varphi } c \xrightarrow{} a'' \xrightarrow{} \Sigma^2 a'$ with $a',a'' \in \cA$.  By the octahedral axiom the two triangles can be combined to the following commutative diagram where each row and column is a triangle.
\[
  \xymatrix @+0.5pc {
    a' \ar^-{\Sigma^{-1}\varphi}[r] \ar@{=}[d] & \Sigma^{-1}c \ar[r] \ar^{\delta}[d] & \Sigma^{-1}a'' \ar[r] \ar[d] & \Sigma a' \ar@{=}[d] \\
    a' \ar^{\delta\Sigma^{-1}\varphi}[r] \ar[d] & a^0 \ar[r] \ar^{\alpha^0}[d] & d \ar[r] \ar[d] & \Sigma a' \ar[d] \\
    0 \ar[r] \ar[d] & a^1 \ar@{=}[r] \ar[d] & a^1 \ar[r] \ar[d] & 0 \ar[d] \\
    \Sigma a' \ar_{\varphi}[r] & c \ar[r] & a'' \ar[r] & \Sigma^2 a'
                    }
\]
The diagram implies $\alpha^0 \circ \delta\Sigma^{-1}\varphi = 0$.  Since $\delta\Sigma^{-1}\varphi$ is a morphism in $\cA$ while $\alpha^0$ is a categorial monomorphism in $\cA$, it follows that $\delta\Sigma^{-1}\varphi = 0$.  Hence the diagram is isomorphic to the following.
\[
  \xymatrix @+0.5pc {
    a' \ar^-{\Sigma^{-1}\varphi}[r] \ar@{=}[d] & \Sigma^{-1}c \ar[r] \ar^{\delta}[d] & \Sigma^{-1}a'' \ar[r] \ar[d] & \Sigma a' \ar@{=}[d] \\
    a' \ar^{0}[r] \ar[d] & a^0 \ar^-{\begin{psmallmatrix} \id \\ 0 \end{psmallmatrix}}[r] \ar^{\alpha^0}[d] & a^0 \oplus \Sigma a' \ar^-{(0,\id)}[r] \ar^{(\alpha^0,0)}[d] & \Sigma a' \ar[d] \\
    0 \ar[r] \ar[d] & a^1 \ar@{=}[r] \ar[d] & a^1 \ar[r] \ar[d] & 0 \ar[d] \\
    \Sigma a' \ar_{\varphi}[r] & c \ar[r] & a'' \ar[r] & \Sigma^2 a'
                    }
\]
The morphism $a^0 \oplus \Sigma a' \xrightarrow{} a^1$ is indeed $(\alpha^0,0)$ since the first component is $\alpha^0$ by commutativity of the middle square while the second component is zero because \ConditionEShort{-1}{\cC}{\cA}.  The third column of the diagram implies $a'' \cong \cone( \alpha^0 ) \oplus \Sigma^2 a'$.  In particular, $c \cong \cone( \alpha^0 )$ is a direct summand of $a'' \in \cA$ whence $c \in \cA$ as desired.
\end{proof}

\section{Abelian subcategories of triangulated categories induced by $w$-orthogonal collections}
\label{sec:abelian}

This section proves Theorems \ref{thm:16c} and \ref{thm:20}, which imply Theorems \ref{thm:A} and \ref{thm:B} from the introduction.

\begin{Setup}
\label{set:abelian}
In this section $\cS = \{ s_1, \ldots, s_e \}$ is a fixed $2$-orthogonal collection in $\cC$.
\end{Setup}

\begin{Remark}
\label{rmk:S_and_A}
The results of Section \ref{sec:exact} apply to $\cA = \langle \cS \rangle$; in particular, $( \langle \cS \rangle,\cE_{ \langle \cS \rangle } )$ is an exact category by Theorem \ref{thm:Dyer}(i), and we will use this without further comment a number of times.  To see that Section \ref{sec:exact} applies, we check Setup \ref{set:exact} for $\cA = \langle \cS \rangle$:  It follows from \cite[lem.\ 2.7]{D} that $\langle \cS \rangle \subseteq \cC$ is an additive subcategory which is closed under extensions by construction; see Definition \ref{def:extension_closure}.  Definition \ref{def:SMS}(ii) implies
\begin{equation}
\label{equ:ConditionEShortforS}
  \mbox{\ConditionEShort{-1}{\cC}{\langle \cS \rangle}.}
\end{equation}
\end{Remark}

\begin{Lemma}
\label{lem:16a}
Let $a^0 \xrightarrow{ \alpha^0 } a^1$ be a morphism in $\langle \cS \rangle$ such that $\cC( s,\alpha^0 )$ is injective for each $s \in \cS$.  Then $\alpha^0$ is a categorical monomorphism in $\langle \cS \rangle$.  
\end{Lemma}

\begin{proof}
We must show that if $a \in \langle \cS \rangle$ then $\langle \cS \rangle( a,\alpha^0 )$ is injective, which is the same as showing that
\begin{equation}
\label{equ:16a_conclusion}
  \mbox{$\cC( a,\alpha^0 )$ is injective.}
\end{equation}
By Definition \ref{def:extension_closure} we can assume $a \in ( \cS )_n$ for a non-negative integer $n$ and will prove \eqref{equ:16a_conclusion} by induction on $n$.  If $n=0$ then $a=0$ so \eqref{equ:16a_conclusion} is trivially true.

For the induction, assume \eqref{equ:16a_conclusion} is true for $a \in ( \cS )_{ n-1 }$ and let $a \in ( \cS )_n \setminus ( \cS )_{ n-1 }$ be given.  Definition \ref{def:extension_closure} gives a triangle $a' \xrightarrow{} a \xrightarrow{} s \xrightarrow{} \Sigma a'$ with $a' \in ( \cS )_{ n-1 }$ and $s \in \cS$.  It induces the following commutative diagram with exact rows.
\[
\vcenter{
  \xymatrix @+0.5pc {
    \cC( \Sigma a',a^0 ) \ar[r] \ar[d] & \cC( s,a^0 ) \ar[r] \ar_{ \cC( s,\alpha^0 ) }[d] & \cC( a,a^0 ) \ar[r] \ar^{ \cC( a,\alpha^0 ) }[d] & \cC( a',a^0 ) \ar^{ \cC( a',\alpha^0 ) }[d] \\
    \cC( \Sigma a',a^1 ) \ar[r] & \cC( s,a^1 ) \ar[r] & \cC( a,a^1 ) \ar[r] & \cC( a',a^1 ) \\
                    }
        }
\]
The first term in each row is zero by Equation \eqref{equ:ConditionEShortforS}, and the map $\cC( s,\alpha^0 )$ is injective by the assumption in the lemma.  Since $a' \in ( \cS )_{ n-1 }$, by induction $\cC( a',\alpha^0 )$ is injective.  Hence $\cC( a,\alpha^0 )$ is injective by the Four Lemma, proving \eqref{equ:16a_conclusion}. 
\end{proof}

\begin{Lemma}
\label{lem:16b}
Each morphism $a^0 \xrightarrow{ \alpha^0 } a^1$ in $\langle \cS \rangle$ can be factorised in $\langle \cS \rangle$ as
\begin{equation}
\label{equ:16c_factorisation}
  \mbox{$a^0 \xrightarrow{\sigma} a \xrightarrow{\iota} a^1$ with }\;
  \left\{
    \begin{array}{l}
      \mbox{$\sigma$ a deflation in $\langle \cS \rangle$,} \\[1.5mm]
      \mbox{$\iota$ a categorial monomorphism in $\langle \cS \rangle$.}
    \end{array}
  \right.
\end{equation}
\end{Lemma}

\begin{proof}
By Definition \ref{def:extension_closure} we can assume $a^0 \in ( \cS )_n$ for a non-negative integer $n$ and will prove the lemma by induction on $n$.  If $n=0$ then $a^0=0$ so the factorisation $a^0 \xrightarrow{} 0 \xrightarrow{} a^1$ can be used in \eqref{equ:16c_factorisation}.

For the induction, assume the lemma holds when $a^0 \in ( \cS )_{ n-1 }$ and let $a^0 \xrightarrow{ \alpha^0 } a^1$ in $\langle \cS \rangle$ with $a^0 \in ( \cS )_n$ be given.  If $\alpha^0$ is a categorical monomorphism in $\langle \cS \rangle$, then the factorisation $a^0 \xrightarrow{ \id } a^0 \xrightarrow{ \alpha^0 } a^1$ can be used in \eqref{equ:16c_factorisation}.  

If $a^0 \xrightarrow{ \alpha^0 } a^1$ is not a categorical monomorphism in $\langle \cS \rangle$, then Lemma \ref{lem:16a} gives an object $s \in \cS$ and a non-zero morphism $s \xrightarrow{ \varphi } a^0$ such that $\alpha^0\varphi = 0$.  We can complete $\varphi$ to a short triangle $s \xrightarrow{ \varphi } a^0 \xrightarrow{ \theta } a'$, and the dual of \cite[lem.\ 2.6]{D} says $a' \in ( \cS )_{ n-1 }$ so the short triangle is a conflation in $\langle \cS \rangle$ whence $\theta$ is a deflation.

We can also complete $\alpha^0$ to a short triangle
$c \xrightarrow{ \gamma } a^0 \xrightarrow{ \alpha^0 } a^1$, and since $\alpha^0\varphi = 0$ we can factorise $\varphi$ through $\gamma$, then use the octahedral axiom to get the following commutative diagram where each row and column is a short triangle.
\[
  \xymatrix @+0.5pc {
    s \ar@{=}[r] \ar[d] & s \ar[r] \ar_{\varphi}[d] & 0 \ar[d] \\
    c \ar^{\gamma}[r] \ar[d] & a^0 \ar^{\alpha^0}[r] \ar_{\theta}[d] & a^1 \ar@{=}[d] \\
    c' \ar[r] & a' \ar_{\alpha'}[r] & a^1 \\
                    }
\]
Since $a' \in ( \cS )_{ n-1 }$, by induction $a' \xrightarrow{ \alpha' } a^1$ can be factorised as $a' \xrightarrow{ \sigma' } a \xrightarrow{ \iota' } a^1$ with $\sigma'$ a deflation and $\iota'$ a categorical monomorphism in $\langle \cS \rangle$.  Hence the factorisation $a^0 \xrightarrow{ \sigma'\theta } a \xrightarrow{ \iota' } a^1$ can be used in \eqref{equ:16c_factorisation} because $\sigma'\theta$ is a deflation in $\langle \cS \rangle$ since $\sigma'$ and $\theta$ are deflations.  
\end{proof}

\begin{Theorem}
\label{thm:16c}
Recall that $\cS$ is a $2$-orthogonal collection in $\cC$.  The category $\langle \cS \rangle$ has the following properties.
\begin{enumerate}
\setlength\itemsep{4pt}

  \item  $\langle \cS \rangle \subseteq \cC$ is a proper abelian subcategory.

  \item  Up to isomorphism the simple objects of $\langle \cS \rangle$ are the objects in $\cS$. 
  
  \item  Each object of $\langle \cS \rangle$ has finite length.

  \item  For $a',a'' \in \langle \cS \rangle$ there is a natural isomorphism 
\[
  \Ext_{ \langle \cS \rangle }( a'',a' )
  \xrightarrow{ \Phi }
  \cC( a'',\Sigma a' ).
\]
  
  \item  The abelian category $\langle \cS \rangle$ has enough projective objects if and only if so does the exact category $( \langle \cS \rangle,\cE_{ \langle \cS \rangle } )$.

\end{enumerate}

\end{Theorem}

\begin{proof}
(i):  We have $\langle \cS \rangle * ( \Sigma \langle \cS \rangle ) \subseteq ( \Sigma \langle \cS \rangle ) * \cS$ by \cite[lem.\ 2.7]{CSP1}.  Note that the methods of \cite{CSP1} apply despite our notion of $w$-orthogonal collections being slightly more general than theirs.  It follows by Lemma \ref{lem:5b} that $\iota$ in Lemma \ref{lem:16b} is an inflation, and hence Lemma \ref{lem:16b} and \cite[ex.\ 8.6(i)]{B} combine to show that $\langle \cS \rangle \subseteq \cC$ is a proper abelian subcategory.

(ii):  On the one hand, let $a$ be a simple object of the abelian category $\langle \cS \rangle$.  In particular $a$ is non-zero, so Definition \ref{def:extension_closure} means that $a \in ( \cS )_n$ for a positive integer $n$ and hence that there is a short triangle $a' \xrightarrow{} a \xrightarrow{ \alpha } s$ with $a' \in ( \cS )_{ n-1 }$.  This is a short exact sequence in $\langle \cS \rangle$ by part (i), so $\alpha$ is an epimorphism in $\langle \cS \rangle$, hence an isomorphism since $a$ is simple.  

On the other hand, let $s$ be in $\cS$ and let $a' \xrightarrow{ \alpha' } s \xrightarrow{ \sigma } a''$ be a short exact sequence in $\langle \cS \rangle$ with $\alpha' \neq 0$.  By part (i) it gives a triangle $\Sigma^{-1} a'' \xrightarrow{} a' \xrightarrow{ \alpha' } s \xrightarrow{ \sigma } a''$, and it follows from \cite[lem.\ 2.6]{D} that $\Sigma^{-1}a'' \in \langle \cS \rangle$.  To show $s$ simple it is enough to show $a'' \cong 0$, which now follows from
\[
  \cC( a'',a'' ) \cong
  \cC( \Sigma^{-1}a'',\Sigma^{-1}a'' ) = 0.
\]
The equality holds by Equation \eqref{equ:ConditionEShortforS} since $\Sigma^{-1}a'',a'' \in \langle \cS \rangle$.

(iii):  By Definition \ref{def:extension_closure} we can assume $a \in ( \cS )_n$ for a non-negative integer $n$ and will prove that $a$ has finite length by induction on $n$, noting that if $n=0$ then $a=0$ does have finite length.  Now assume each $a \in ( \cS )_{ n-1 }$ has finite length and let $a \in ( \cS )_n \setminus ( \cS )_{ n-1 }$ be given.  Definition \ref{def:extension_closure} gives a triangle $a' \xrightarrow{} a \xrightarrow{} s \xrightarrow{} \Sigma a'$ with $a' \in ( \cS )_{ n-1 }$ and $s \in \cS$, hence by part (i) a short exact sequence $a' \xrightarrow{} a \xrightarrow{} s$ in $\langle \cS \rangle$ where $a'$ has finite length by induction while $s$ is simple by part (ii).

(iv), (v):  We know that $( \langle \cS \rangle,\cE_{ \langle \cS \rangle } )$ is an exact category, see Remark \ref{rmk:S_and_A}.  Part (i) says that $\langle \cS \rangle$ is an abelian category whose short exact sequences are the diagrams in $\cE_{ \langle \cS \rangle }$.  This immediately implies (v), and it gives a natural isomorphism
\[
  \Ext_{ \langle \cS \rangle }( a'',a' ) \cong
  \Ext_{ (\langle \cS \rangle,\cE_{ \langle \cS \rangle } ) }( a'',a' )
\]
which can be combined with Theorem \ref{thm:Dyer}(ii) to give (iv).  
\end{proof}

\begin{Theorem}
\label{thm:20}
Recall that $\cS$ is a $2$-orthogonal collection in $\cC$.  Assume that each object of $\langle \cS \rangle$ has a $\Sigma \langle \cS \rangle$-envelope.  Then the following hold.
\begin{enumerate}
\setlength\itemsep{4pt}

  \item  $\langle \cS \rangle \simeq \mod\,A$ for a finite dimensional $k$-algebra $A$.  

  \item  If $k$ is algebraically closed, then the Gabriel quiver of $A$ has vertices the elements of $\cS = \{ s_1, \ldots, s_e \}$ and $\dim_k \cC( s_j,\Sigma s_i )$ arrows from $s_i$ to $s_j$. 

\end{enumerate}
\end{Theorem}

\begin{proof}
(i):  Observe that $\langle \cS \rangle$ is a Krull--Schmidt category because it is an additive subcategory of $\cC$, see Remark \ref{rmk:S_and_A}, while $\cC$ is Krull--Schmidt since it satisfies Setup \ref{set:blanket}.  Hence, if $p$ is an indecomposable projective object of $\langle \cS \rangle$, then the algebra $\End_{ \langle \cS \rangle }( p )$ is local.  It follows by \cite[prop.\ 3.7]{Krause} that there is a unique maximal non-trivial subobject $r$ of $p$.  The reason \cite[prop.\ 3.7]{Krause} applies is that $\langle \cS \rangle$ is essentially small while each object has finite length by Theorem \ref{thm:16c}(iii).  The quotient $p/r$ is simple, hence isomorphic to some $s \in \cS$ by Theorem \ref{thm:16c}(ii), so there is an epimorphism $p \xrightarrow{} s$.  It follows that $p$ is a projective cover of $s$ by \cite[lem.\ 3.6]{Krause}, hence determined up to isomorphism by $s$ by \cite[cor.\ 3.5]{Krause}.

The abelian category $\langle \cS \rangle$ has enough projective objects by Proposition \ref{pro:9} and Theorem \ref{thm:16c}(v).  Combining with the previous paragraph shows that up to isomorphism, the indecomposable projective objects of $\langle \cS \rangle$ are precisely given by projective covers $\{ p_1, \ldots, p_e \}$ of the objects in $\cS = \{ s_1, \ldots, s_e \}$ whence $g = p_1 \oplus \cdots p_e$ is a projective generator.  It follows that $A = \End_{ \langle \cS \rangle }(g)$ satisfies $\langle \cS \rangle \simeq \mod\,A$ by \cite[III.3, proposition]{AQM}.

(ii):  Combine part (i) with Theorem \ref{thm:16c}(ii)+(iv) and \cite[prop.\ III.1.14]{ARS}.
\end{proof}

\begin{Remark}
\label{rmk:Thm_A_does_not_extend}
Theorems \ref{thm:16c} and \ref{thm:20} do not extend to $1$-orthogonal collections.  For instance, if $Q$ is a Dynkin quiver of type $A_e$, then the simple modules in $\mod\,kQ$ give a $1$-orthogonal collection $\cS$ in the negative cluster category $\cC_{ -1 }( Q ) = \cD^{\b}( \mod\,kQ ) / \tau\Sigma^2$, see \cite[lem.\ 4.2(1)]{CS2}, and we even have $\langle \cS \rangle = \cC_{ -1 }( Q )$.  However, $\cC_{ -1 }( Q )$ is not abelian for $e \geqslant 2$.
\end{Remark}

\begin{Remark}
\label{rmk:AB}
The assumption in Theorem \ref{thm:A} implies Setup \ref{set:abelian} because a $w$-orthogonal collection is $2$-orthogonal when $w \geqslant 2$.  Hence Theorem \ref{thm:A} follows from Theorem \ref{thm:16c}(i)-(iii).  Similarly, the assumption in Theorem \ref{thm:B} implies Setup \ref{set:abelian}, and it implies the assumption in Theorem \ref{thm:20} by Lemma \ref{lem:functorial_finiteness}.  Hence Theorem \ref{thm:B} follows from Theorem \ref{thm:20}.
\end{Remark}

\section{Lemmas on torsion pairs}
\label{sec:torsion}

This section shows several lemmas on torsion pairs in the abelian category $\langle \cS \rangle$.

\begin{Setup}
\label{set:torsion}
In this section $\cS = \{ s_1, \ldots, s_e \}$ is a fixed $2$-orthogonal collection in $\cC$, and $( \cT,\cF )$ is a torsion pair in the extension closure $\langle \cS \rangle$, which is abelian by Theorem \ref{thm:16c}(i).
\end{Setup}

\begin{Lemma}
\label{lem:23}
Let $a^0 \xrightarrow{ \alpha^0 } a^1$ be a morphism in $\langle \cS \rangle$ with kernel $\Ker \alpha^0$ and cokernel $\Coker \alpha^0$.  If $a^0 \xrightarrow{ \alpha^0 } a^1 \xrightarrow{} c$ is a short triangle, then there is a short triangle $\Sigma \Ker \alpha^0 \xrightarrow{} c \xrightarrow{} \Coker \alpha^0$.  
\end{Lemma}

\begin{proof}
There are short exact sequences $\Ker \alpha^0 \xrightarrow{} a^0 \xrightarrow{ \sigma } \Image \alpha^0$ and $\Image \alpha^0 \xrightarrow{ \iota } a^1 \xrightarrow{} \Coker \alpha^0$
in $\langle \cS \rangle$ with $\alpha^0 = \iota\sigma$.  Since $\langle \cS \rangle \subseteq \cC$ is a proper abelian subcategory, the short exact sequences are short triangles.  By the octahedral axiom they can be combined to the following commutative diagram where each row and column is a short triangle.  
\[
  \xymatrix @+0.5pc {
    \Ker \alpha^0 \ar[r] \ar[d] & 0 \ar[r] \ar[d] & \Sigma \Ker \alpha^0 \ar[d] \\
    a^0 \ar^{\alpha^0}[r] \ar_{ \sigma }[d] & a^1 \ar[r] \ar@{=}[d] & c \ar[d] \\
    \Image \alpha^0 \ar_{ \iota }[r] & a^1 \ar[r] & \Coker \alpha^0 \\
                    }
\]
The right column is the short triangle claimed in the lemma.
\end{proof}

\begin{Lemma}
\label{lem:TF_closed_under_extensions}
Let $i$ be an integer, $\cV \subseteq \langle \cS \rangle$ an additive subcategory which is closed under extensions in the abelian category $\langle \cS \rangle$.  Then $\Sigma^i \cV$ is closed under extensions in the triangulated category $\cC$.

In particular, $\Sigma^i \cT$ and $\Sigma^i \cF$ are closed under extensions in the triangulated category $\cC$.  
\end{Lemma}

\begin{proof}
Since $\Sigma$ is an automorphism, it is enough to show that $\cV$ is closed under extensions in $\cC$, so let $v' \xrightarrow{} e \xrightarrow{} v''$ be a short triangle with $v',v'' \in \cV$.  In particular, $v',v'' \in \langle \cS \rangle$ so we have $e \in \langle \cS \rangle$.  Hence the short triangle is a short exact sequence in $\langle \cS \rangle$ whence $e \in \cV$.
\end{proof}

\begin{Lemma}
\label{lem:24_and_25_small_version}
The following are satisfied.
\begin{enumerate}
\setlength\itemsep{4pt}

  \item  $( \Sigma^{ -1 }\cT ) * \cF \subseteq
  \cF * ( \Sigma^{ -1 }\cT )$.

  \item  $\cF * ( \Sigma^{ -1 }\cT ) \subseteq \cC$ is an additive subcategory closed under extensions.   

\end{enumerate}
\end{Lemma}

\begin{proof}
(i):  Given $e \in ( \Sigma^{ -1 }\cT ) * \cF$ there is a short triangle $\Sigma^{ -1 }t \xrightarrow{} e \xrightarrow{} f$ with $t \in \cT$, $f \in \cF$, hence a short triangle $f \xrightarrow{ \varphi } t \xrightarrow{} \Sigma e$.  By Lemma \ref{lem:23} this gives a short triangle
\begin{equation}
\label{equ:24_and_25_small_version_2}
  \Sigma \Ker \varphi \xrightarrow{} \Sigma e \xrightarrow{} \Coker \varphi
\end{equation}
where kernel and cokernel are taken in $\langle \cS \rangle$.  But $\cF$ is a torsion free class and $\cT$ a torsion class in $\langle \cS \rangle$, so the subobject $\Ker \varphi$ of $f$ is in $\cF$ and the quotient object $\Coker \varphi$ of $t$ is in $\cT$, whence \eqref{equ:24_and_25_small_version_2} shows $\Sigma e \in ( \Sigma \cF ) * \cT$, that is, $e \in \cF * ( \Sigma^{ -1 }\cT )$ as desired.

(ii):  Definition \ref{def:star}(ii) gives that $\cF * \Sigma^{ -1 }\cT \subseteq \cC$ is an additive subcategory since we have $\cC( \cF,\Sigma^{ -1 }\cT ) = 0$ by Equation \eqref{equ:ConditionEShortforS}.  Closure under extensions can be shown as follows.
\[
  \cF * ( \Sigma^{ -1 }\cT ) * \cF * ( \Sigma^{ -1 }\cT ) \stackrel{ \rm (a) }{ \subseteq }
  \cF * \cF * ( \Sigma^{ -1 }\cT ) * ( \Sigma^{ -1 }\cT ) \stackrel{ \rm (b) }{ \subseteq }
  \cF * ( \Sigma^{ -1 }\cT )
\]
We have used associativity of $*$ which holds by \cite[lem.\ 1.3.10]{BBD}.  The inclusion (a) is by part (i) of the lemma while (b) is by Lemma \ref{lem:TF_closed_under_extensions}.
\end{proof}

\begin{Lemma}
\label{lem:f_to_T}
Assume that $\cT$ is closed under subobjects in $\langle \cS \rangle$ and that $f \in \cF$ is a simple object of $\langle \cS \rangle$.  Then $\cC( f,\cT ) = 0$. 
\end{Lemma}

\begin{proof}
Let $t \in \cT$ be given.  A non-zero morphism $f \xrightarrow{} t$ would be monic in $\langle \cS \rangle$ since $f$ is simple in $\langle \cS \rangle$.  Hence $f$ would be a subobject of $t$, but $\cT$ is closed under subobjects, so this would show $f \in \cT$ contradicting that $f$ is a non-zero object of $\cF$.  
\end{proof}

\begin{Lemma}
\label{lem:35}
Assume that $\cT$ is closed under subobjects in $\langle \cS \rangle$ and that
\begin{itemize}
\setlength\itemsep{4pt}

  \item  $f \in \cF$ is a simple object of $\langle \cS \rangle$,
  
  \item  There is a triangle
\begin{equation}
\label{equ:35_100}
  \Sigma^{ -1 }t \xrightarrow{ \tau } f \xrightarrow{ \varphi } g \xrightarrow{} t
\end{equation}
where $\tau$ is a $( \Sigma^{ -1 }\cT )$-cover.

\end{itemize}
Then $g \in \cF$.  
\end{Lemma}

\begin{proof}
The triangle \eqref{equ:35_100} implies $g \in \langle \cS \rangle$ since $\langle \cS \rangle$ is closed under extensions.  The torsion pair $( \cT,\cF )$ provides a short exact sequence $t' \xrightarrow{ \iota } g \xrightarrow{ \sigma } f'$ in $\langle \cS \rangle$ with $t' \in \cT$, $f' \in \cF$, hence a triangle 
\begin{equation}
\label{equ:35_2}
  t' \xrightarrow{ \iota } g \xrightarrow{ \sigma } f' \xrightarrow{} \Sigma t'.
\end{equation}
Since $g \in \langle \cS \rangle$, the short triangle $f \xrightarrow{ \varphi } g \xrightarrow{} t$ arising from \eqref{equ:35_100} is a short exact sequence in $\langle \cS \rangle$ whence $\varphi \neq 0$ because $f \neq 0$.  From $\varphi \neq 0$ follows $\sigma\varphi \neq 0$ since if $\sigma\varphi = 0$ then we could factorise $f \xrightarrow{ \varphi } g$ as $f \xrightarrow{ \theta } t' \xrightarrow{ \iota } g$, but $\theta = 0$ by Lemma \ref{lem:f_to_T}.

By the octahedral axiom the triangles \eqref{equ:35_100} and \eqref{equ:35_2} can be combined to the following commutative diagram where each row and column is a triangle.
\[
  \xymatrix @+0.5pc {
    \Sigma^{ -1 }t' \ar[r] \ar_{ \mu }[d] & 0 \ar[r] \ar[d] & t' \ar@{=}[r] \ar^{ \iota }[d] & t' \ar^{ \Sigma \mu }[d] \\
    \Sigma^{ -1 }t \ar^{ \tau }[r] \ar_{ \pi }[d] & f \ar^{ \varphi }[r] \ar@{=}[d] & g \ar[r] \ar^{ \sigma }[d] & t \ar[d] \\
    \Sigma^{ -1 }c \ar_{ \widetilde{ \tau } }[r] \ar[d] & f \ar_{ \sigma\varphi }[r] \ar[d] & f' \ar[r] \ar[d] & c \ar[d] \\
    t' \ar[r] & 0 \ar[r] & \Sigma t' \ar@{=}[r] & \Sigma t' \\
                    }
\]
Since $f$ is simple in $\langle \cS \rangle$ we have $f \in \cS$ up to isomorphism by Theorem \ref{thm:16c}(ii).  Since $f' \in \langle \cS \rangle$ and $\sigma\varphi \neq 0$, the dual of \cite[lem.\ 2.6]{D} hence shows $c \in \langle \cS \rangle$.  The last column of the diagram therefore gives a short exact sequence
\begin{equation}
\label{equ:35_1}
  t' \xrightarrow{ \Sigma \mu } t \xrightarrow{} c
\end{equation}
in $\langle \cS \rangle$ whence $c \in \cT$ since $t \in \cT$ while $\cT$ is closed under quotients because it is a torsion class. 

It follows that $\Sigma^{ -1 }c \in \Sigma^{ -1 }\cT$ so the morphism $\Sigma^{ -1 }c \xrightarrow{ \widetilde{ \tau } } f$ can be factorised as $\Sigma^{ -1 }c \xrightarrow{ \rho } \Sigma^{ -1 }t \xrightarrow{ \tau } f$ through the $( \Sigma^{ -1 }\cT )$-cover $\tau$.  Hence $\tau\rho\pi = \widetilde{\tau}\pi = \tau$, and this shows that $\rho\pi$ is an automorphism because $\tau$ is right minimal since it is a cover.  But then $\rho\pi \circ \mu = \rho \circ \pi\mu = \rho \circ 0 = 0$ implies $\mu = 0$.  Hence $\Sigma \mu = 0$ so the short exact sequence \eqref{equ:35_1} implies $t' \cong 0$ whence the triangle \eqref{equ:35_2} implies $g \cong f'$.  But $f' \in \cF$ so $g \in \cF$ as claimed.  
\end{proof}

\begin{Lemma}
\label{lem:36}
Assume that $\cT$ is closed under subobjects in $\langle \cS \rangle$ and that
\begin{itemize}
\setlength\itemsep{4pt}

  \item  $f_1, f_2 \in \cF$ are simple objects of $\langle \cS \rangle$,
  
  \item  There are triangles
\[
  \Sigma^{ -1 }t_i \xrightarrow{ \tau_i } f_i \xrightarrow{ \varphi_i } g_i \xrightarrow{} t_i
\]
for $i \in \{ 1,2 \}$ where the $\tau_i$ are $( \Sigma^{ -1 }\cT )$-covers.

\end{itemize}
Then there is an isomorphism $\cC( g_1,g_2 ) \xrightarrow{ \Psi } \cC( f_1,f_2 )$, which respects compositions in the case where $f_1 = f_2$ and the triangles are identical.
\end{Lemma}

\begin{proof}
The mode of proof is due to \cite[proof of thm.\ 6.2]{D2}.  The two triangles in the lemma combine to the following diagram with an exact row and an exact column.
\[
  \xymatrix @+0.5pc {
    & & \cC( t_1,g_2 ) \ar[d] & \\
    & & \cC( g_1,g_2 ) \ar^{ \varphi_1^* }[d] & \\
    \cC( f_1,\Sigma^{ -1 }t_2 ) \ar_-{ (\tau_2)_* }[r] & \cC( f_1,f_2 ) \ar_{ (\varphi_2)_* }[r] & \cC( f_1,g_2 ) \ar[r] \ar^{ \tau_1^* }[d] & \cC( f_1,t_2 ) \\
    & & \cC( \Sigma^{ -1 }t_1,g_2 ) & \\
                    }
\]
We show that the four end terms are zero.  
\begin{itemize}
\setlength\itemsep{4pt}

  \item  $\cC( f_1,\Sigma^{ -1 }t_2 ) = 0$ holds by Equation \eqref{equ:ConditionEShortforS}.
  
  \item  $\cC( f_1,t_2 ) = 0$ by Lemma \ref{lem:f_to_T}.
  
  \item  $\cC( t_1,g_2 ) = 0$ since $t_1 \in \cT$ by assumption while $g_2 \in \cF$ by Lemma \ref{lem:35}.
  
  \item  $\cC( \Sigma^{ -1 }t_1,g_2 ) = 0$ by the triangulated Wakamatsu Lemma, \cite[lem.\ 2.1]{JorARSub}, since $g_2$ is the mapping cone of a $( \Sigma^{ -1 }\cT )$-cover while $\Sigma^{ -1 }\cT \subseteq \cC$ is closed under extensions by Lemma \ref{lem:TF_closed_under_extensions}.

\end{itemize}
The diagram hence shows that $\varphi_1^*$ and $( \varphi_2 )_*$ are bijective.  This means that in the square
\[
  \xymatrix @+0.5pc {
    f_1 \ar^{ \varphi_1 }[r] \ar_{ \psi }[d] & g_1 \ar^{ \gamma }[d] \\
    f_2 \ar_{ \varphi_2 }[r] & g_2 \lefteqn{,} \\                    }
\]
if one of the vertical morphisms is given, then there is a unique choice of the other making the square commutative.  This implies that $\Psi( \gamma ) = \psi$ defines a map $\Psi$ with the properties claimed in the lemma.
\end{proof}

For the next proof we remind the reader that by Theorem \ref{thm:16c}(ii)+(iii), the objects of $\cS$ are simple in $\langle \cS \rangle$ and each object of $\langle \cS \rangle$ has finite length.

\begin{Lemma}
\label{lem:S_prime}
Let $\cS' \subseteq \cS$ be a subset.  Then $\langle \cS' \rangle \subseteq \langle \cS \rangle$ is the additive subcategory whose objects are those with all composition factors in $\cS'$.
\end{Lemma}

\begin{proof}
Let $\cV \subseteq \langle \cS \rangle$ be the additive subcategory whose objects are those with all composition factors in $\cS'$.  Lemma \ref{lem:TF_closed_under_extensions} implies that $\cV \subseteq \cC$ is closed under extensions, and it is clear that $\cS' \subseteq \cV$, so we have $\langle \cS' \rangle \subseteq \cV$.  To prove the opposite inclusion $\cV \subseteq \langle \cS' \rangle$, we can assume that $v \in \cV$ has length $n$ for a non-negative integer $n$ and will prove $v \in \langle \cS' \rangle$ by induction on $n$, noting that if $n=0$ then $v=0$ is in $\langle \cS' \rangle$.  Now assume that each $v \in \cV$ of length $n-1$ satisfies $v \in \langle \cS' \rangle$ and let $v \in \cV$ of length $n$ be given.  There is a short exact sequence $v' \xrightarrow{} v \xrightarrow{} s'$ with $v' \in \cV$ of length $n-1$ and $s' \in \cS'$.  This is a short triangle where $v' \in \langle \cS' \rangle$ by induction so $v \in \langle \cS' \rangle$ follows.  
\end{proof}

\section{Left tilting of $w$-orthogonal collections}
\label{sec:left_tilting}

This section proves Theorem \ref{thm:37}, which implies the first part of Theorem \ref{thm:C} from the introduction.

\begin{Setup}
\label{set:left_tilting}
In this section the following are fixed.
\begin{enumerate}
\setlength\itemsep{4pt}

  \item  $w \geqslant 2$ is an integer.
  
  \item  $k$ is still a field, $\cC$ an essentially small $k$-linear $\Hom$-finite triangulated category with split idempotents, and we now also assume that $\cC$ is $(-w)$-Calabi--Yau.

  \item  $\cS = \{ s_1, \ldots, s_e \}$ is a fixed $w$-orthogonal collection in $\cC$ and $\cS' \subseteq \cS$ is a subset.  We have the extension closures $\langle \cS' \rangle \subseteq \langle \cS \rangle$. 

  \item  Each $s \in \cS \setminus \cS'$ is assumed to have a $( \Sigma^{ -1 }\langle \cS' \rangle )$-cover.

\end{enumerate}
\end{Setup}

\begin{Remark}
The results of Sections \ref{sec:abelian} and \ref{sec:torsion} apply to $\langle \cS \rangle$, and Setup \ref{set:left_tilting}(iv) implies that the left tilt $L_{ \cS' }( \cS )$ is defined, see Definition \ref{def:tilting}(i).
\end{Remark}

\begin{Theorem}
\label{thm:37}
The left tilt $L_{ \cS' }( \cS )$ and the classes $\cT = \langle \cS' \rangle$, $\cF = \langle \cS' \rangle^{ \perp } \cap \langle \cS \rangle$ have the following properties.\begin{enumerate}
\setlength\itemsep{4pt}

  \item  $L_{ \cS' }( \cS )$ is a $w$-orthogonal collection.

  \item  If $\cS$ is a $w$-simple minded system, then so is $L_{ \cS' }( \cS )$.

  \item  $( \cT,\cF )$ is a torsion pair in the abelian category $\langle \cS \rangle$ and $\cT$ is closed under subobjects in $\langle \cS \rangle$.

  \item  $( \cF,\Sigma^{-1}\cT )$ is a torsion pair in the abelian category $\langle L_{ \cS' }( \cS ) \rangle$ and $\Sigma^{-1}\cT$ is closed under quotient objects in $\langle L_{ \cS' }( \cS ) \rangle$.

\end{enumerate}
\end{Theorem}

The four parts of the theorem will be proved separately.

\begin{proof}[Proof of Theorem \ref{thm:37}(iii)]  Lemma \ref{lem:S_prime} says that $\cT \subseteq \langle \cS \rangle$ is the additive subcategory of objects with composition factors isomorphic to objects in $\cS'$.  In particular, it is closed under subobjects.  It is also a torsion class because each object of $\langle \cS \rangle$ has finite length by Theorem \ref{thm:16c}(iii).  The corresponding torsion free class is the right perpendicular of $\cT$ taken in $\langle \cS \rangle$, which is precisely $\cF$.  
\end{proof}

\begin{Lemma}
\label{lem:sec_5_after}
The classes $\cT$ and $\cF$ from Theorem \ref{thm:37} satisfy the following.
\begin{enumerate}
\setlength\itemsep{4pt}

  \item  If $s \in \cS'$ then $s \in \cT$.
  
  \item  If $s \in \cS \setminus \cS'$ then $s \in \cF$. 
  
  \item  If $s \in \cS \setminus \cS'$ then $L_{ \cS' }( s ) \in \cF$.

\end{enumerate}
\end{Lemma}

\begin{proof}
(i):  This holds since $\cT = \langle \cS' \rangle$.  

(ii):  If $s \in \cS \setminus \cS'$ then $\cC( \cS',s ) = 0$ by Definition \ref{def:SMS}(i) whence $\cC( \langle \cS' \rangle,s ) = 0$, so $s \in \langle \cS' \rangle^{ \perp } \cap \langle \cS \rangle = \cF$.  

(iii):  Theorem \ref{thm:16c}(ii), Theorem \ref{thm:37}(iii), and part (ii) imply that Lemma \ref{lem:35} can applied to the triangle \eqref{equ:L_triangle} to give part (iii).
\end{proof}

\begin{proof}[Proof of Theorem \ref{thm:37}(i)]
We will split into four cases to show that $L_{ \cS' }( \cS ) = \{ L_{ \cS' }( s_1 ), \ldots, L_{ \cS' }( s_e ) \}$ satisfies Definition \ref{def:SMS}(i)+(ii), using that by assumption, so does $\cS$.
\begin{itemize}
\setlength\itemsep{4pt}

  \item  If $s_i,s_j \in \cS'$ then it is sufficient to note that
\[
  \cC\big( L_{ \cS' }( s_i ),\Sigma^{ \ell }L_{ \cS' }( s_j ) \big)
  = \cC( \Sigma^{ -1 }s_i,\Sigma^{ \ell }\Sigma^{ -1 }s_j )
  \cong \cC( s_i,\Sigma^{ \ell }s_j ).
\]

  \item  If $s_i,s_j \in \cS \setminus \cS'$ then Theorem \ref{thm:16c}(ii), Theorem \ref{thm:37}(iii), and Lemma \ref{lem:sec_5_after}(ii) imply that Lemma \ref{lem:36} can applied to the triangle \eqref{equ:L_triangle} to get
\[
  \cC\big( L_{ \cS' }( s_i ),L_{ \cS' }( s_j ) \big)
  \cong \cC( s_i,s_j ).
\]
This is sufficient for Definition \ref{def:SMS}(i).
For Definition \ref{def:SMS}(ii) note that $L_{ \cS' }( s_i ), L_{ \cS' }( s_j ) \in \cF \subseteq \langle \cS \rangle$ by Lemma \ref{lem:sec_5_after}(iii), so if $\ell \in \{ -w+1, \ldots, -1 \}$ then
\[
  \cC\big( L_{ \cS' }( s_i ),\Sigma^{ \ell }L_{ \cS' }( s_j ) \big) = 0
\]
follows from Definition \ref{def:SMS}(ii) for $\cS$.

  \item  If $s_i \in \cS'$, $s_j \in \cS \setminus \cS'$ then 
\[
  \cC\big( L_{ \cS' }( s_i ),\Sigma^{ \ell }L_{ \cS' }( s_j ) \big) \cong
  \cC\big( \Sigma^{ -1 }s_i,\Sigma^{ \ell }L_{ \cS' }( s_j ) \big) \cong
  \cC\big( s_i,\Sigma^{ \ell+1 }L_{ \cS' }( s_j ) \big) = (*).
\]
For Definition \ref{def:SMS}(i) note that if $\ell = 0$ then $(*) = 0$ by the triangulated Wakamatsu Lemma, \cite[lem.\ 2.1]{JorARSub}, since $s_i \in \cT$ by Lemma \ref{lem:sec_5_after}(i) while $L_{ \cS' }( s_j )$ is the mapping cone of a $( \Sigma^{ -1 }\cT )$-cover and $\Sigma^{ -1 }\cT \subseteq \cC$ is closed under extensions by Lemma \ref{lem:TF_closed_under_extensions}.  For Definition \ref{def:SMS}(ii) note that $L_{ \cS' }( s_j ) \in \cF \subseteq \langle \cS \rangle$ by Lemma \ref{lem:sec_5_after}(iii) so if $\ell \in \{ -w+1, \ldots, -2 \}$ then $(*) = 0$ follows from Definition \ref{def:SMS}(ii) for $\cS$.  If $\ell = -1$ then $(*) = \cC \big( s_i,L_{ \cS' }( s_j ) \big) = 0$ since $s_i \in \cT$ and $L_{ \cS' }( s_j ) \in \cF$ by Lemma \ref{lem:sec_5_after}(i)+(iii).

  \item  If $s_i \in \cS \setminus \cS'$, $s_j \in \cS'$ then 
\[
  \cC\big( L_{ \cS' }( s_i ),\Sigma^{ \ell }L_{ \cS' }( s_j ) \big) \cong
  \cC\big( L_{ \cS' }( s_i ),\Sigma^{ \ell }\Sigma^{ -1 }s_j \big) \cong
  \cC\big( L_{ \cS' }( s_i ),\Sigma^{ \ell-1 }s_j \big) = (**).
\]
We have $L_{ \cS' }( s_i ) \in \cF \subseteq \langle \cS \rangle$ by Lemma \ref{lem:sec_5_after}(iii).  For Definition \ref{def:SMS}(i) note that if $\ell = 0$ then $(**) = 0$ follows from Definition \ref{def:SMS}(i) for $\cS$.  For Definition \ref{def:SMS}(ii), if $\ell \in \{ -w+2, \ldots, -1 \}$ then $(**) = 0$ follows from Definition \ref{def:SMS}(ii) for $\cS$.  If $\ell = -w+1$ then
\[
  (**) = \cC\big( L_{ \cS' }( s_i ),\Sigma^{ -w }s_j \big) \cong
  \dual\cC\big( s_j,L_{ \cS' }( s_i ) \big) = 0,
\]
where the isomorphism holds because $\cC$ is $(-w)$-Calabi--Yau and the last equality holds since $s_j \in \cT$ and $L_{ \cS' }( s_i ) \in \cF$ by Lemma \ref{lem:sec_5_after}(i)+(iii). \qedhere
\end{itemize}
\end{proof}

\begin{proof}[Proof of Theorem \ref{thm:37}(iv)]
The proof will be divided into two steps.

Step 1: This step will prove
\begin{align}
\label{equ:37_4_100}
  \cF & \subseteq \langle L_{ \cS' }( \cS ) \rangle, \\
\label{equ:37_4_101}
  \Sigma^{ -1 }\cT & \subseteq \langle L_{ \cS' }( \cS ) \rangle.
\end{align}
To prove Equation \eqref{equ:37_4_101}, observe that $\Sigma^{ -1 }\cS' \subseteq L_{ \cS' }( \cS )$ by Definition \ref{def:tilting}(i) whence $\langle \Sigma^{ -1 }\cS' \rangle \subseteq \langle L_{ \cS' }( \cS ) \rangle$, so $\Sigma^{ -1 }\cT = \Sigma^{ -1 }\langle \cS' \rangle = \langle \Sigma^{ -1 }\cS' \rangle \subseteq \langle L_{ \cS' }( \cS ) \rangle$.  

We will prove Equation \eqref{equ:37_4_100} by showing
\begin{equation}
\label{equ:F_in_L_1}
  f \in \cF
  \Rightarrow
  f \in \langle L_{ \cS' }( \cS ) \rangle
\end{equation}
by induction on $n = \length_{ \langle \cS \rangle }( f )$, which makes sense by Theorem \ref{thm:16c}(i)+(iii) since $\cF \subseteq \langle \cS \rangle$.  If $n=0$ then $f=0$ so \eqref{equ:F_in_L_1} is trivially true.

For the induction, assume \eqref{equ:F_in_L_1} is true for $\length_{ \langle \cS \rangle }( f ) = n-1$ and let $f \in \cF$ satisfy $\length_{ \langle \cS \rangle }( f ) = n$.  There is a short exact sequence
\begin{equation}
\label{equ:F_in_L_2}
  f' \xrightarrow{} f \xrightarrow{} u
\end{equation}
in $\langle \cS \rangle$ with $\length_{ \langle \cS \rangle }( f' ) = n-1$ and $u$ simple in $\langle \cS \rangle$.  Recall that $( \cT,\cF )$ is a torsion pair in $\langle \cS \rangle$ by Theorem \ref{thm:37}(iii).  Hence $f \in \cF$ implies  $f' \in \cF$ so
\begin{equation}
\label{equ:F_in_L_4}
  f' \in \langle L_{ \cS' }( \cS ) \rangle
\end{equation}
by induction.  By Theorem \ref{thm:16c}(ii) we can assume $u \in \cS$.   

If $u \in \cS \setminus \cS'$ then Definition \ref{def:tilting}(i) gives a short triangle $\Sigma^{ -1 }t_u \xrightarrow{} u \xrightarrow{} L_{ \cS' }( u )$ with $t_u \in \cT$.  Each end term is in $\langle L_{ \cS' }( \cS ) \rangle$, the first by Equation \eqref{equ:37_4_101} and the second by definition.   Hence $u \in \langle L_{ \cS' }( \cS ) \rangle$ so Equations \eqref{equ:F_in_L_2} and \eqref{equ:F_in_L_4} imply $f \in \langle L_{ \cS' }( \cS ) \rangle$.

If $u \in \cS' \subseteq \cT$ then the short exact sequence \eqref{equ:F_in_L_2} cannot split because $f \in \cF$.  Hence \eqref{equ:F_in_L_2} induces a triangle $\Sigma^{ -1 }u \xrightarrow{ \varphi } f' \xrightarrow{} f \xrightarrow{} u$ with $\varphi \neq 0$.  By Definition \ref{def:tilting}(i) it reads $L_{ \cS' }( u ) \xrightarrow{ \varphi } f' \xrightarrow{} f \xrightarrow{} u$.  Hence $f \in \langle L_{ \cS' }( \cS ) \rangle$ by Equation \eqref{equ:F_in_L_4} and the dual of \cite[lem.\ 2.6]{D}, which applies because $L_{ \cS' }( \cS )$ is a $w$-orthogonal collection by Theorem \ref{thm:37}(i).

Step 2: This step will complete the proof.  Note that $\langle L_{ \cS' }( \cS ) \rangle \subseteq \cC$ is a proper abelian subcategory by Theorem \ref{thm:16c}(i) and Theorem \ref{thm:37}(i).

We have the inclusions \eqref{equ:37_4_100} and \eqref{equ:37_4_101}, and $\cF$ and $\Sigma^{ -1 }\cT$ are full subcategories closed under isomorphisms in $\cC$.  To show that $( \cF,\Sigma^{-1}\cT )$ is a torsion pair in $\langle L_{ \cS' }( \cS ) \rangle$, it is hence sufficient to show that it satisfies axioms (IV) and (V) in \cite[sec.\ 1]{Dickson}, and axiom (V) says $\cC( \cF,\Sigma^{ -1 }\cT ) = 0$ which is true by Equation \eqref{equ:ConditionEShortforS}.

For axiom (IV), stating the existence of fundamental sequences, note that Definition \ref{def:tilting}(i) and Lemma \ref{lem:sec_5_after}(i)+(iii) imply that $L_{ \cS' }( \cS ) \subseteq \cF \cup ( \Sigma^{ -1 }\cT )$, in particular $L_{ \cS' }( \cS ) \subseteq \cF * ( \Sigma^{ -1 }\cT )$.  Theorem \ref{thm:37}(iii) and Lemma \ref{lem:24_and_25_small_version}(ii) give that $\cF * ( \Sigma^{ -1 }\cT )$ is closed under extensions in $\cC$, so it follows that $\langle L_{ \cS' }( \cS ) \rangle \subseteq \cF * ( \Sigma^{ -1 }\cT )$.  Each $a \in \langle L_{ \cS' }( \cS ) \rangle$ hence permits a short triangle $f \xrightarrow{} a \xrightarrow{} \Sigma^{ -1 }t$ with $f \in \cF$ and $\Sigma^{ -1 }t \in \Sigma^{ -1 }\cT$.  This is a short exact sequence in $\langle L_{ \cS' }( \cS ) \rangle \subseteq \cC$ because of Equations \eqref{equ:37_4_100} and \eqref{equ:37_4_101}, showing that axiom (IV) in \cite[sec.\ 1]{Dickson} holds.

Finally, Lemma \ref{lem:S_prime} gives that $\Sigma^{ -1 }\cT = \Sigma^{ -1 }\langle \cS' \rangle = \langle \Sigma^{ -1 }\cS' \rangle = \langle L_{ \cS' }( \cS' ) \rangle$ is the additive subcategory of $\langle L_{ \cS' }( \cS ) \rangle$ consisting of objects with composition factors isomorphic to objects in $L_{ \cS' }( \cS' )$.  In particular, it is closed under quotient objects in $\langle L_{ \cS' }( \cS ) \rangle$. 
\end{proof}

\begin{proof}[Proof of Theorem \ref{thm:37}(ii)]
Assume that $\cS$ is a $w$-simple minded system.  Since $L_{ \cS' }( \cS )$ is a $w$-orthogonal collection by Theorem \ref{thm:37}(i), it remains to prove that it satisfies Definition \ref{def:SMS}(iii)'.

Observe that we have
\begin{align}
\label{equ:37_4_10}
  \langle \cS \rangle & = \cT * \cF, \\
\label{equ:37_4_11}
  \langle L_{ \cS' }( \cS ) \rangle & = \cF * \Sigma^{ -1 }\cT.
\end{align}
In each case, the equation holds because we have a torsion pair in a proper abelian subcategory; see Theorem \ref{thm:37}(iii)+(iv).  Combining Definition \ref{def:SMS}(iii)' for $\cS$ with Equation \eqref{equ:37_4_10} gives
\begin{align}
\nonumber
  \cC & =
  \langle \cS \rangle * \Sigma^{ -1 } \langle \cS \rangle * \cdots * \Sigma^{ -w+1 }\langle \cS \rangle \\
\nonumber  
  & = \cT * \cF * ( \Sigma^{ -1 } \cT ) * ( \Sigma^{ -1 } \cF ) * \cdots * ( \Sigma^{ -w+1 } \cT ) * ( \Sigma^{ -w+1 } \cF ) \\
\label{equ:37_4_1}
  & = \cT * \cY
\end{align}
with
\begin{equation}
\label{equ:37_4_1000}
  \cY = \cF * ( \Sigma^{ -1 } \cT ) * ( \Sigma^{ -1 } \cF ) * \cdots * ( \Sigma^{ -w+1 } \cT ) * ( \Sigma^{ -w+1 } \cF ).
\end{equation}
In this formula, on the right hand side, none of the subcategories have non-zero morphisms to any of the subsequent ones; this follows from $( \cT,\cF )$ being a torsion pair in $\langle \cS \rangle$ and Definition \ref{def:SMS}(ii).  Hence $\cY$ is closed under summands by repeated application of \cite[prop.\ 2.1(1)]{IY}.  

Now let $c \in \cC$ be given.  Let
\begin{equation}
\label{equ:37_4_3}	
  c' \xrightarrow{} c \xrightarrow{ \gamma } \Sigma^{ -w }t
\end{equation}
be a short triangle with $\gamma$ a $( \Sigma^{ -w }\cT )$-envelope, which exists by Lemma \ref{lem:functorial_finiteness}.  The triangulated Wakamatsu Lemma, dual to \cite[lem.\ 2.1]{JorARSub}, says $\cC( c',\Sigma^{ -w }\cT ) = 0$ because $\Sigma^{ -w }\cT \subseteq \cC$ is closed under extensions by Lemma \ref{lem:TF_closed_under_extensions}.  Since $\cC$ is $(-w)$-Calabi--Yau, it follows that
\begin{equation}
\label{equ:37_4_2}	
  \cC( \cT,c' ) \cong \dual\!\cC( c',\Sigma^{ -w }\cT ) = 0.
\end{equation}
By Equation \eqref{equ:37_4_1} there is a short triangle $t \xrightarrow{ \varphi } c' \xrightarrow{} y$ with $t \in \cT$ and $y \in \cY$.  Equation \eqref{equ:37_4_2} implies $\varphi = 0$ so  $c'$ is a direct summand of $y$.  Since $\cY$ is closed under direct summands this gives $c' \in \cY$.    But then Equations \eqref{equ:37_4_3}, \eqref{equ:37_4_1000}, and \eqref{equ:37_4_11} imply
\begin{align*}
  c & \in \cY * ( \Sigma^{ -w } \cT ) \\
  & = \cF * ( \Sigma^{ -1 } \cT ) * ( \Sigma^{ -1 } \cF ) * \cdots * ( \Sigma^{ -w+1 } \cT ) * ( \Sigma^{ -w+1 } \cF ) * ( \Sigma^{ -w } \cT ) \\
  & = \langle L_{ \cS' }( \cS ) \rangle * \Sigma^{ -1 }\langle L_{ \cS' }( \cS ) \rangle * \cdots * \Sigma^{ -w+1 }\langle L_{ \cS' }( \cS ) \rangle,
\end{align*}
proving that $L_{ \cS' }( \cS )$ satisfies Definition \ref{def:SMS}(iii)'.
\end{proof}

\section{Right tilting of $w$-orthogonal collections}
\label{sec:right_tilting}

This section states Theorem \ref{thm:37_prime}, which is the dual of Theorem \ref{thm:37} and implies the second part of Theorem \ref{thm:C} from the introduction.  The proof of Theorem \ref{thm:37_prime} is omitted as it is dual to the proof of Theorem \ref{thm:37}.

\begin{Setup}
\label{set:right_tilting}
In this section the following are fixed.
\begin{enumerate}
\setlength\itemsep{4pt}

  \item  $w \geqslant 2$ is an integer.
  
  \item  $k$ is still a field, $\cC$ an essentially small $k$-linear $\Hom$-finite triangulated category with split idempotents, and we assume that $\cC$ is $(-w)$-Calabi--Yau.

  \item  $\cS = \{ s_1, \ldots, s_e \}$ is a fixed $w$-orthogonal collection in $\cC$ and $\cS'' \subseteq \cS$ is a subset.  We have the extension closures $\langle \cS'' \rangle \subseteq \langle \cS \rangle$. 

  \item  Each $s \in \cS \setminus \cS''$ is assumed to have a $( \Sigma\langle \cS'' \rangle )$-envelope.

\end{enumerate}
\end{Setup}

\begin{Remark}
Setup \ref{set:right_tilting}(iv) implies that the right tilt $R_{ \cS'' }( \cS )$ is defined, see Definition \ref{def:tilting}(ii).
\end{Remark}

\begin{Theorem}
\label{thm:37_prime}
The right tilt $R_{ \cS'' }( \cS )$ and the classes $\cU = {}^{ \perp }\langle \cS'' \rangle \cap \langle \cS \rangle$, $\cG = \langle \cS'' \rangle$ have the following properties.\begin{enumerate}
\setlength\itemsep{4pt}

  \item  $R_{ \cS'' }( \cS )$ is a $w$-orthogonal collection.

  \item  If $\cS$ is a $w$-simple minded system, then so is $R_{ \cS'' }( \cS )$.

  \item  $( \cU,\cG )$ is a torsion pair in the abelian category $\langle \cS \rangle$ and $\cG$ is closed under quotient objects in $\langle \cS \rangle$.

  \item  $( \Sigma\cG,\cU )$ is torsion pair in the abelian category $\langle R_{ \cS'' }( \cS ) \rangle$ and $\Sigma\cG$ is closed under subobjects in $\langle R_{ \cS'' }( \cS ) \rangle$.

\end{enumerate}
\end{Theorem}

We continue by showing that left and right tilting are inverse operations.

\begin{Proposition}
\label{pro:left_tilting_v_right_tilting}
Assume Setups \ref{set:left_tilting} and \ref{set:right_tilting} whence the left tilt $L_{ \cS' }( \cS )$ and the right tilt $R_{ \cS'' }( \cS )$ are defined.  Then the following are satisfied.
\begin{enumerate}
\setlength\itemsep{4pt}

  \item  The right tilt $R_{ \Sigma^{ -1 }\cS' } \big( L_{ \cS' }( \cS ) \big)$ is defined and equals $\cS$ up to isomorphism.

  \item  The left tilt $L_{ \Sigma\cS'' } \big( R_{ \cS'' }( \cS ) \big)$ is defined and equals $\cS$ up to isomorphism.

\end{enumerate}
\end{Proposition}

\begin{proof}
(i):  Given $s \in \cS \setminus \cS'$ the object $L_{ \cS' }( s )$ is defined by a triangle
\[
  \Sigma^{ -1 }t_s \xrightarrow{ \tau } s \xrightarrow{ \sigma } L_{ \cS' }( s ) \xrightarrow{ \gamma } t_s
\]  
with $\tau$ a $\Sigma^{ -1 }\langle \cS' \rangle$-cover, see Definition \ref{def:tilting}(i).  We claim that $\gamma$ is an $\langle \cS' \rangle$-envelope, that is, a $\Sigma( \Sigma^{ -1 }\langle \cS' \rangle )$-envelope.  This shows that the right tilt in part (i) is defined and that
\begin{equation}
\label{equ:left_tilting_v_right_tilting_1}
  \mbox{$R_{ \Sigma^{ -1 }\cS' }\big( L_{ \cS' }( s ) \big) \cong s$ \; for \; $s \in \cS \setminus \cS'$,}
\end{equation}
see Definition \ref{def:tilting}(ii).  

To see that $\gamma$ is an $\langle \cS' \rangle$-preenvelope, let $L_{ \cS' }( s ) \xrightarrow{ \lambda } t$ be a morphism with $t \in \langle \cS' \rangle$.  Since $\cC( s,\langle \cS' \rangle ) = 0$ follows from Definition \ref{def:SMS}(i), we have $\lambda\sigma = 0$ so $\lambda$ can be factorised through $\gamma$.

To see that $\gamma$ is left minimal, it is enough by \cite[lem.\ 3.12(b)]{F} to see that $\Sigma\tau$ is in the radical, hence enough to see that $\tau$ is in the radical.  If not, then $\tau$ would have a component that was an isomorphism.  This would imply that $s$ was a direct summand of $\Sigma^{ -1 }t_s$, because $s$ is indecomposable since $\cC( s,s )$ is a skew field by Definition \ref{def:SMS}(i).  Hence we would have $s \in \Sigma^{ -1 }\langle \cS' \rangle \subseteq \Sigma^{ -1 }\langle \cS \rangle$.  We also have $s \in \langle \cS \rangle$ so Equation \eqref{equ:ConditionEShortforS} would give $\cC( s,s ) = 0$ which is false since $\cC( s,s )$ is a skew field.  

Finally, observe that Definition \ref{def:tilting} gives
\[
  \mbox{$R_{ \Sigma^{ -1 }\cS' } \big( L_{ \cS' }( s' ) \big) = R_{ \Sigma^{ -1 }\cS' }( \Sigma^{ -1 }s' ) = 
  \Sigma\Sigma^{ -1 }s' = s'$ \; for \; $s' \in \cS'$.}
\]
Combining with Equation \eqref{equ:left_tilting_v_right_tilting_1} completes the proof of part (i).  

(ii):  Follows by dual arguments.  
\end{proof}

\begin{Remark}
\label{rmk:C}
The assumptions in Theorem \ref{thm:C} clearly imply Setup \ref{set:left_tilting}(i)-(iii) and Setup \ref{set:right_tilting}(i)-(iii), and they imply Setup \ref{set:left_tilting}(iv) and Setup \ref{set:right_tilting}(iv) by Lemma \ref{lem:functorial_finiteness}.  Hence Theorem \ref{thm:C} follows from Theorems \ref{thm:37}(ii)-(iv) and \ref{thm:37_prime}(ii)-(iv).
\end{Remark}

\section{Example: Negative cluster categories of Dynkin type $A$}
\label{sec:Type_A}

This section applies Theorems \ref{thm:A}, \ref{thm:B}, and \ref{thm:C} from the introduction to the negative cluster categories of Dynkin type $A$ first studied in \cite{CS3}.

\begin{Setup}
\label{set:Type_A}
In this section the following are fixed.
\begin{itemize}
\setlength\itemsep{4pt}

  \item  $e \geqslant 1$ and $w \geqslant 2$ are integers.
  
  \item  $Q$ is a Dynkin quiver of type $A_e$.
  
  \item  $\cC = \cC_{ -w }( Q ) = \cD^{ \b }( \mod\,kQ )/ \tau\Sigma^{ w+1 }$ is the $(-w)$-cluster category of Dynkin type $A_e$.
  

\end{itemize}
\end{Setup}

\begin{Remark}
\label{rmk:Type_A}
The category $\cC$ satisfies Setup \ref{set:blanket}, see
\cite[sec.\ 1]{BMRRT} 
and \cite[sec.\ 4]{KellerOrbit}, 
and it is $(-w)$-Calabi--Yau by \cite[sec.\ 8.4]{KellerOrbit}.  In particular, Theorems \ref{thm:A} through \ref{thm:C} from the introduction apply.
\end{Remark}

\begin{Remark}
[A combinatorial model of $\cC$]
\label{rmk:CPS2}
The following combinatorial model of $\cC$ is due to \cite{CS3}.  Set $N = ( w+1 )( e+1 ) - 2$ and let $P$ be an $N$-gon with vertices $\{ 0, \ldots, N-1 \}$ numbered in the clockwise direction.  The vertices have a cyclic order defined by their positions on $P$, and arithmetic of vertices is done modulo $N$.  A {\em diagonal} is a set $s = \{ \dt{s},\ddt{s} \}$ of two distinct vertices, which are allowed to be consecutive.  It has {\em end points} $\dt{s}$, $\ddt{s}$, and $s$ is {\em admissible} if it divides $P$ into two subpolygons whose numbers of vertices are divisible by $w+1$.  Diagonals $s = \{ \dt{s},\ddt{s} \}$ and $s' = \{ \dt{s}',\ddt{s}' \}$ are said to {\em cross} if $\dt{s} < \dt{s}' < \ddt{s} < \ddt{s}' < \dt{s}$ or $\dt{s} < \ddt{s}' < \ddt{s} < \dt{s}' < \dt{s}$.  The combinatorial model has the following key properties.
\begin{enumerate}
\setlength\itemsep{4pt}

  \item  There is a bijection
\[
  \{\mbox{admissible diagonals}\} \longleftrightarrow
  \{\mbox{indecomposable objects of $\cC$}\},
\]
see \cite[sec.\ 10.1]{CSP2}.  We will use ``admissible diagonal'' as synonymous with ``indecomposable object''.  

  \item  If $s$, $s'$ are admissible diagonals, then the following statements are equivalent, see \cite[sec.\ 10.1]{CSP2}.
\begin{itemize}
\setlength\itemsep{4pt}

  \item  There is an irreducible morphism $s \xrightarrow{} s'$.

  \item  There are vertices $\dt{s}$, $\ddt{s}$ such that $s = \{ \dt{s},\ddt{s} \}$ and $s' = \{ \dt{s},\ddt{s}+w+1 \}$.

\end{itemize}

  \item  The suspension functor acts on admissible diagonals by
\[
  \Sigma \{ \dt{s},\ddt{s} \} = \{ \dt{s}+1,\ddt{s}+1 \},
\]
see \cite[sec.\ 10.1]{CSP2}.  

  \item  We have
\[
  \{ \mbox{$w$-simple minded systems in $\cC$} \}
  =
  \left\{ \cS
    \left| 
      \begin{array}{l}
        \mbox{$\cS$ is a set of $e$ admissible diagonals} \\[1.5mm]
        \mbox{which are pairwise non-crossing and} \\[1.5mm]
        \mbox{do not share any endpoints}
      \end{array}
    \right.
  \right\},
\]
see \cite[thm.\ 6.5]{CS3}, \cite[prop.\ 2.13]{CSP1}.
\end{enumerate}
\end{Remark}

The following result is due to \cite[cor.\ 6.3]{CS3} and \cite[cor.\ 10.6 and fig.\ 11]{CSP2}.

\begin{Proposition}
\label{pro:Ext}
Let $s' = \{ \dt{s}',\ddt{s}' \}$ and $s$ be admissible diagonals.  Then
\[
  \cC( \Sigma^{ -1 }s',s )
  =
  \left\{
    \begin{array}{cl}
      k & \mbox{if (i), (ii), or (iii) below holds,} \\[1.5mm]
      0 & \mbox{otherwise,}
    \end{array}
  \right.
\]
where (i), (ii), and (iii) are the following conditions.  
\begin{enumerate}
\setlength\itemsep{4pt}

  \item  $s$ and $s'$ do not cross and do not share an end point, but $s$ has $\dt{s}'-1$ or $\ddt{s}'-1$ as an end point.

  \item  $s$ and $s'$ cross and $s = \{ \dt{s}'-i(w+1),\ddt{s}'-j(w+1) \}$ for certain integers $i,j$. 
  
  \item  $s = \{ \dt{s}'-1,\ddt{s}'-1 \}$.

\end{enumerate}
If (i), (ii), or (iii) holds then there is a triangle $\Sigma^{ -1 }s' \xrightarrow{ \delta } s \xrightarrow{} e \xrightarrow{} s'$ with $\delta \neq 0$, and $e$ has indecomposable summands given by the dotted diagonals in the relevant part of Figure \ref{fig:Ext}.
\begin{figure}
\begingroup
\[
  \begin{tikzpicture}[scale=1.5]

    \begin{scope}[shift={(-4.2,0)}]
      \draw (90:1.9cm) node {\rm (i)};      
      \node[name=s, shape=circle, minimum size=4.5cm, draw] {}; 
      \draw[thick] (90-6*360/21:1.5cm) to (90+6*360/21:1.5cm);
      \draw[thick] (90-2*360/21:1.5cm) to (90+5*360/21:1.5cm);
      \draw[very thick,dotted] (90-2*360/21:1.5cm) to (90-6*360/21:1.5cm);    
      \draw (90-10.5*360/21:0.5cm) node {$s$};      
      \draw (90+1.5*360/21:0.95cm) node {$s'$};      
      \draw (90-4*360/21:1.05cm) node {$e$};      
    \end{scope}

    \begin{scope}[shift={(0,0)}]
      \draw (90:1.9cm) node {\rm (ii)};      
      \node[name=s, shape=circle, minimum size=4.5cm, draw] {}; 
      \draw[thick] (90-6*360/21:1.5cm) to (90+6*360/21:1.5cm);
      \draw[thick] (90+0*360/21:1.5cm) to (90+12*360/21:1.5cm);
      \draw[very thick,dotted] (90-6*360/21:1.5cm) to (90+0*360/21:1.5cm);    
      \draw[very thick,dotted] (90+6*360/21:1.5cm) to (90+12*360/21:1.5cm);    
      \draw (90+5.85*360/21:0.9cm) node {$s$};      
      \draw (90+0.3*360/21:0.9cm) node {$s'$};      
      \draw (90-3*360/21:1.2cm) node {$e_1$};      
      \draw (90+9*360/21:1.1cm) node {$e_2$};      
    \end{scope}

    \begin{scope}[shift={(4.2,0)}]
      \draw (90:1.9cm) node {\rm (iii)};      
      \node[name=s, shape=circle, minimum size=4.5cm, draw] {}; 
      \draw[thick] (90-6*360/21:1.5cm) to (90+6*360/21:1.5cm);
      \draw[thick] (90-7*360/21:1.5cm) to (90+5*360/21:1.5cm);
      \draw (90+7.2*360/21:0.95cm) node {$s$};      
      \draw (90+4.65*360/21:0.8cm) node {$s'$};      
    \end{scope}

  \end{tikzpicture} 
\]
\endgroup
\caption{The three parts of this figure correspond to Proposition \ref{pro:Ext}, parts (i)--(iii).  The indecomposable objects $s'$ and $s$ permit a triangle $\Sigma^{ -1 }s' \xrightarrow{ \delta } s \xrightarrow{} e \xrightarrow{} s'$ with $\delta \neq 0$.  In (i) and (ii) the object $e$ has indecomposable summands given by the dotted diagonals, and in (iii) we have $e=0$ since $\delta$ is an isomorphism.}
\label{fig:Ext}
\end{figure}
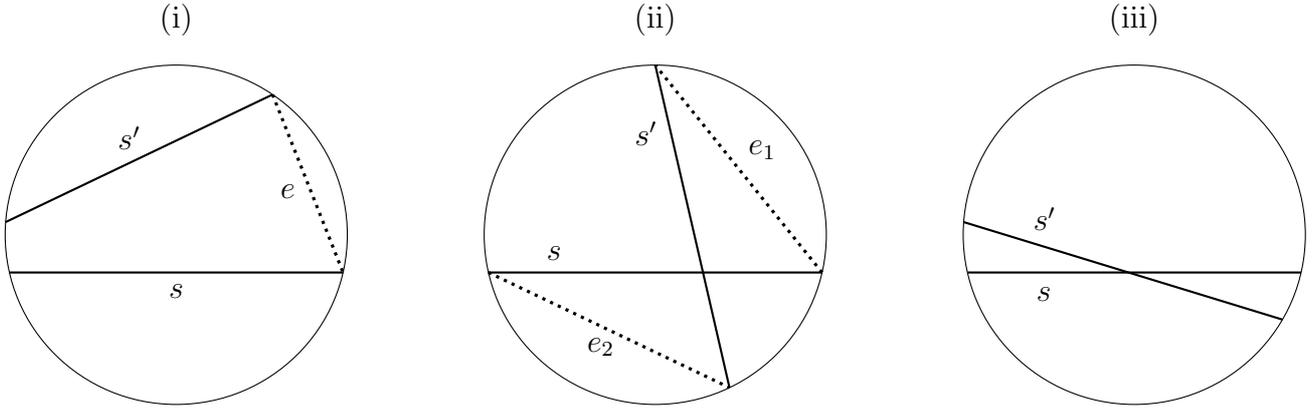
\end{Proposition}

\begin{Lemma}
\label{lem:type_A_cover}
Let $s'$, $s$ be indecomposable objects of $\cC$.
\begin{enumerate}
\setlength\itemsep{4pt}

  \item  If $\cC( \Sigma^{ -1 }s',s ) \neq 0$ then each non-zero morphism $\Sigma^{ -1 }s' \xrightarrow{ \delta } s$ is a $( \Sigma^{ -1 }\langle s' \rangle )$-cover.

  \item  If $\cC( \Sigma^{ -1 }s',s ) = 0$ then $0 \xrightarrow{} s$ is a $( \Sigma^{ -1 }\langle s' \rangle )$-cover.

\end{enumerate}
\end{Lemma}

\begin{proof}
We have $\cC( s',\Sigma s' ) \cong \cC( \Sigma^{ -1 }s',s' ) = 0$ where the equality is by Proposition \ref{pro:Ext}.  Hence $\add( s' )$ is closed under extensions whence $\langle s' \rangle = \add( s' )$.  It follows that $\Sigma^{ -1 }\langle s' \rangle = \add( \Sigma^{ -1 }s' )$, so a $( \Sigma^{ -1 }\langle s' \rangle )$-cover is the same thing as an $\add( \Sigma^{ -1 }s' )$-cover.  This implies part (ii) of the lemma immediately, and it implies part (i) because $\cC( \Sigma^{ -1 }s',s ) \neq 0$ means $\cC( \Sigma^{ -1 }s',s ) = k$ by Proposition \ref{pro:Ext}.
\end{proof}

The following proposition gives a combinatorial description of left tilting of a $w$-simple minded system at an indecomposable object.  We leave it to the reader to formulate a dual description of right tilting.

\begin{figure}
\begingroup
\[
  \begin{tikzpicture}[scale=1.5]

    \begin{scope}[shift={(-4.2,0)}]
      \draw (90:1.9cm) node {\rm (i)};      
      \node[name=s, shape=circle, minimum size=4.5cm, draw] {}; 
      \draw[thick] (90-6*360/21:1.5cm) to (90+6*360/21:1.5cm);
      \draw[thick] (90-2*360/21:1.5cm) to (90+5*360/21:1.5cm);
      \draw[very thick,dotted] (90-2*360/21:1.5cm) to (90-6*360/21:1.5cm);    
      \draw (90-10.5*360/21:0.5cm) node {$s$};      
      \draw (90+1.5*360/21:0.95cm) node {$s'$};      
      \draw (90-3.8*360/21:0.70cm) node {$L_{ \{ s' \} }( s )$};      
    \end{scope}

    \begin{scope}[shift={(0,0)}]
      \draw (90:1.9cm) node {\rm (ii)};      
      \node[name=s, shape=circle, minimum size=4.5cm, draw] {}; 
      \draw[thick] (90-8*360/21:1.5cm) to (90+8*360/21:1.5cm);
      \draw[thick] (90-2*360/21:1.5cm) to (90+5*360/21:1.5cm);
      \draw (90-10.5*360/21:0.85cm) node {$L_{ \{ s' \} }( s ) = s$};      
      \draw (90+1.5*360/21:0.95cm) node {$s'$};      
    \end{scope}

    \begin{scope}[shift={(4.2,0)}]
      \draw (90:1.9cm) node {\rm (iii)};      
      \node[name=s, shape=circle, minimum size=4.5cm, draw] {}; 
      \draw[very thick,dotted] (90-1*360/21:1.5cm) to (90+6*360/21:1.5cm);
      \draw[thick] (90-2*360/21:1.5cm) to (90+5*360/21:1.5cm);
      \draw (90+3.6*360/21:1.15cm) node {$s'$};      
      \draw (0.15,0) node {$L_{ \{ s' \} }( s' ) = \Sigma^{ -1 }s'$};      
    \end{scope}

  \end{tikzpicture} 
\]
\endgroup
\caption{The three parts of this figure correspond to Proposition \ref{pro:left_tilting}, parts (i)--(iii).  The objects $s'$, $s$ are elements of a $w$-simple minded system $\cS$.  In (i) and (iii) the left tilts $L_{ \{ s' \} }( s )$ and $L_{ \{ s' \} }( s' )$ are given by the dotted diagonals, and in (ii) we have $L_{ \{ s' \} }( s ) = s$.}
\label{fig:L}
\end{figure}
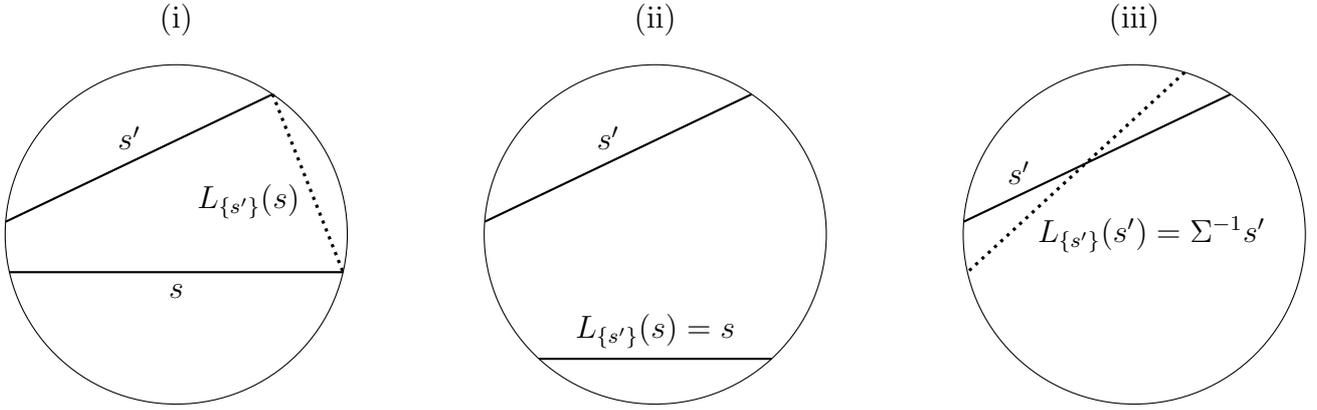
\begin{figure}
\begingroup
\[
  \vcenter{
  \xymatrix @+0.5pc @!0 {
    57 \ar[dr] && 08 \ar[dr] && {\color{green} 13} \ar[dr] && 46 \ar[dr] && {\color{red} 79} \ar[dr] && 02 \ar[dr] \\
    & 05 \ar[dr] \ar[ur] && 38 \ar[dr] \ar[ur] && {\color{red} 16} \ar[dr] \ar[ur] && 49 \ar[dr] \ar[ur] && 27 \ar[dr] \ar[ur] && 05 \\
    02 \ar[ur] && {\color{red} 35} \ar[ur] && 68 \ar[ur] && {\color{green} 19} \ar[ur] && 24 \ar[ur] && 57 \ar[ur] \\
                        }
          }
\]
\endgroup
\caption{This figure is part of Example \ref{exa:Type_A_Theorems_A_and_B} and shows the Auslander--Reiten quiver of $\cC_{-2}( A_3 )$, the $(-2)$-cluster category of Dynkin type $A_3$.  Coloured vertices show the $2$-simple minded system $\cS$ from Figure \ref{fig:SMS}.  The objects of $\cS$ are red, and the remaining indecomposable objects of $\langle \cS \rangle$ are green.}
\label{fig:SMS2}
\end{figure}
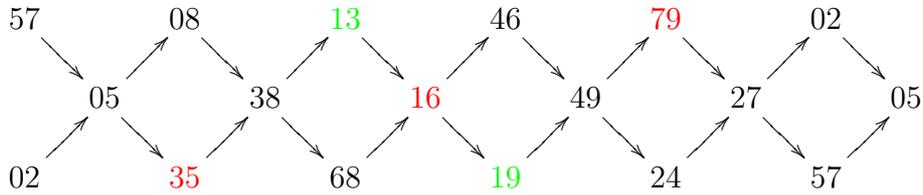
\begin{figure}
\begingroup
\[
  \begin{tikzpicture}
      \node[name=s, shape=regular polygon, regular polygon sides=10, minimum size=4.5cm, draw] {}; 
      \draw (-0*360/10:2.5cm) node {$0$};      
      \draw (-1*360/10:2.5cm) node {$1$};      
      \draw (-2*360/10:2.5cm) node {$2$};            
      \draw (-3*360/10:2.5cm) node {$3$};      
      \draw (-4*360/10:2.5cm) node {$4$};      
      \draw (-5*360/10:2.5cm) node {$5$};            
      \draw (-6*360/10:2.5cm) node {$6$};            
      \draw (-7*360/10:2.5cm) node {$7$};      
      \draw (-8*360/10:2.5cm) node {$8$};      
      \draw (-9*360/10:2.5cm) node {$9$};            
      \draw[red,thick] (-3*360/10:2.25cm) to (-5*360/10:2.25cm);
      \draw[red,thick] (-1*360/10:2.25cm) to (-6*360/10:2.25cm);
      \draw[red,thick] (-7*360/10:2.25cm) to (-9*360/10:2.25cm);
  \end{tikzpicture} 
\]
\endgroup
\caption{This figure is part of Example \ref{exa:Type_A_Theorems_A_and_B} and shows a $2$-simple minded system $\cS$ in $\cC_{-2}( A_3 )$.}
\label{fig:SMS}
\end{figure}
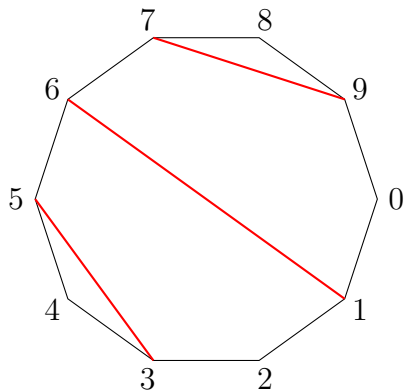
\begin{figure}
\begingroup
\[
\begin{tikzpicture}

\def \polygonradius {1.5cm}
\def \radiusa {8cm}
\def \radiusb {4.69cm}
\def \displacementangle {0}
\def \margin {0.1cm}

\foreach \navn/\vinkel/\diagonaleta/\diagonaletb/\diagonaltoa/\diagonaltob in {a/0/1/4/10/7,b/-36/8/1/7/4,c/-72/5/8/4/1,d/-108/2/5/1/8,e/-144/9/2/8/5,f/-180/6/9/5/2,g/-216/3/6/2/9,h/-252/10/3/9/6,i/-288/7/10/6/3,j/-324/4/7/3/10}
{
  \node[shape=regular polygon,regular polygon sides=10,minimum size=\polygonradius,draw] (\navn) at (\vinkel:\radiusa) {}; 
  \draw (\navn.corner \diagonaleta) to (\navn.corner \diagonaletb);
  \draw (\navn.corner \diagonaltoa) to (\navn.corner \diagonaltob);
}

\foreach \navn/\vinkel/\diagonaleta/\diagonaletb/\diagonaltoa/\diagonaltob in {u/-54-\displacementangle/2/5/10/7,v/-\displacementangle-270/9/2/7/4,x/-\displacementangle-126/6/9/4/1,y/-\displacementangle-342/3/6/1/8,z/-\displacementangle-198/10/3/8/5}
{
  \node[shape=regular polygon,regular polygon sides=10,minimum size=\polygonradius,draw] (\navn) at (\vinkel:\radiusb) {}; 
  \draw (\navn.corner \diagonaleta) to (\navn.corner \diagonaletb);
  \draw (\navn.corner \diagonaltoa) to (\navn.corner \diagonaltob);
}

\foreach \navna/\navnb/\udvinkel/\indvinkel in {a/b/-90-0*36-8/90-1*36+8,b/c/-90-1*36-8/90-2*36+8,c/d/-90-2*36-8/90-3*36+8,d/e/-90-3*36-8/90-4*36+8,e/f/-90-4*36-8/90-5*36+8,f/g/-90-5*36-8/90-6*36+8,g/h/-90-6*36-8/90-7*36+8,h/i/-90-7*36-8/90-8*36+8,i/j/-90-8*36-8/90-9*36+8,j/a/-90-9*36-8/90-10*36+8}
{
  \draw [shorten >=\margin,shorten <=\margin,->,>=latex] (\navna) to [out=\udvinkel,in=\indvinkel] (\navnb);
}

\foreach \navna/\navnb in {a/u,u/d,f/u,u/i,b/v,v/e,g/v,v/j,c/x,x/f,h/x,x/a,d/y,y/g,i/y,y/b,e/z,z/h,j/z,z/c}
{
  \draw [shorten >=\margin,shorten <=\margin,->,>=latex] (\navna) to (\navnb);
}

\end{tikzpicture}
\]
\endgroup
\caption{This figure is part of Example \ref{exa:Type_A_Theorem_C}.  The arrows show the left tilts at indecomposable objects of the $3$-simple minded systems in $\cC_{ -3 }( A_2 )$, the $(-3)$-cluster category of Dynkin type $A_2$.}
\label{fig:LA}
\end{figure}
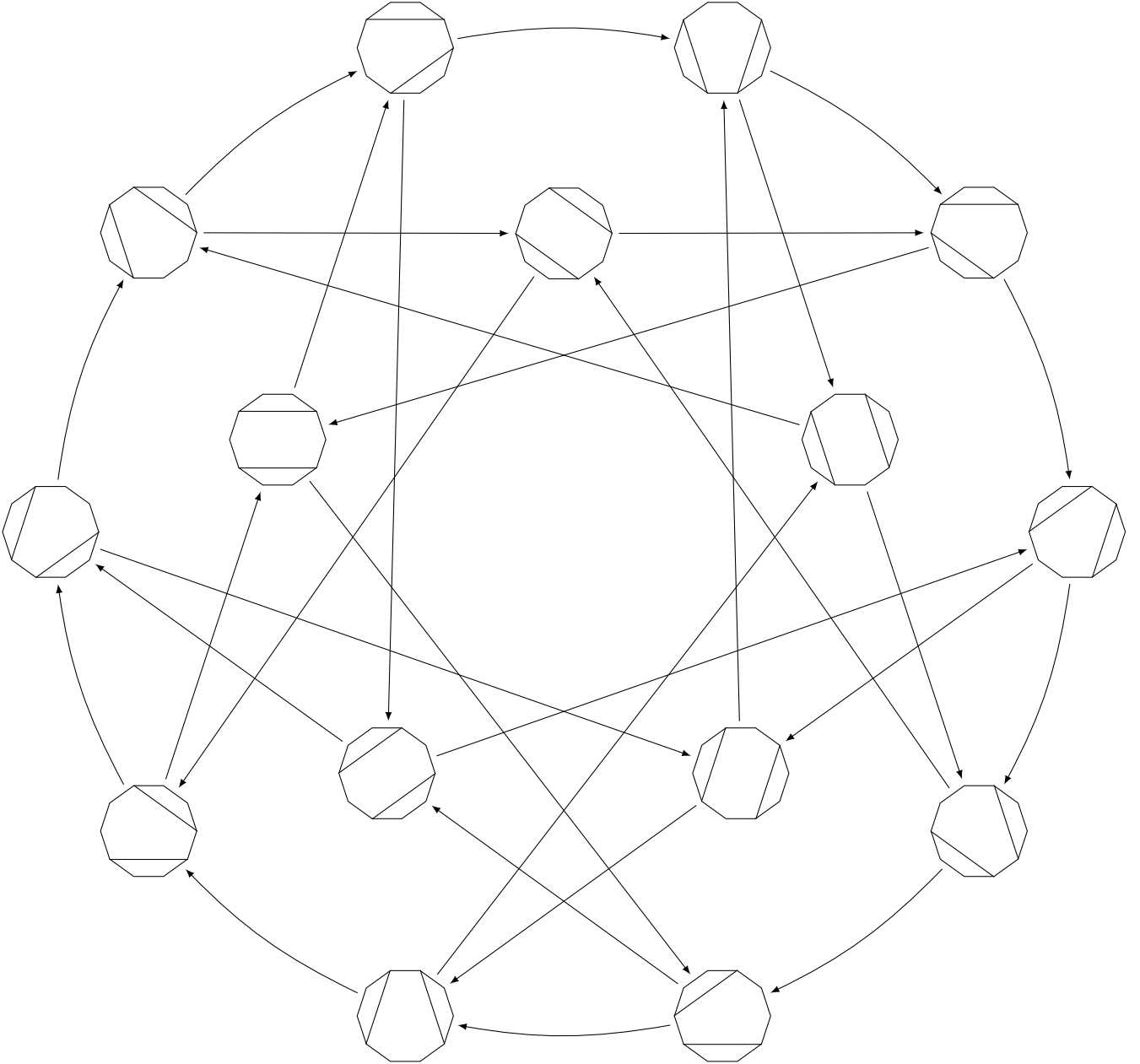

\begin{Proposition}
\label{pro:left_tilting}
Let $s' = \{ \dt{ s }',\ddt{ s }' \}$ and $s$ be elements of a $w$-simple minded system $\cS$ in $\cC$.  Then the left tilt $L_{ \{ s' \} }( s )$ is given as follows.
\begin{enumerate}
\setlength\itemsep{4pt}

  \item  If $s \in \cS \setminus \{ s' \}$ has $\dt{s}'-1$ or $\ddt{s}'-1$ as an end point, then $L_{ \{ s' \} }( s )$ is given by the dotted diagonal in Figure \ref{fig:L}(i).

  \item  If $s \in \cS \setminus \{ s' \}$ does not have $\dt{s}'-1$ or $\ddt{s}'-1$ as an end point, then $L_{ \{ s' \} }( s ) = s$, see Figure \ref{fig:L}(ii).

  \item  $L_{ \{ s' \} }( s' ) = \Sigma^{ -1 }s' = \{ \dt{ s }'-1,\ddt{ s }'-1 \}$, see Figure \ref{fig:L}(iii).

\end{enumerate}
\end{Proposition}

\begin{proof}
By Lemma \ref{lem:functorial_finiteness} the subcategory $\Sigma^{ -1 }\langle s' \rangle$ is functorially finite, so the left tilt is defined; see Definition \ref{def:tilting}(i).  

Let $s \in \cS \setminus \{ s' \}$ be given.  Then $s$ and $s'$ do not cross and do not share an end point since they are both in the $w$-simple minded system $\cS$; see Remark \ref{rmk:CPS2}(iv). 

For part (i), assume that $s$ has $\dt{s}'-1$ or $\ddt{s}'-1$ as an end point.  Proposition \ref{pro:Ext}(i) says that $\cC( \Sigma^{ -1 }s',s ) = k$.  The proposition also gives a triangle $\Sigma^{ -1 }s' \xrightarrow{ \delta } s \xrightarrow{} e \xrightarrow{} s'$ with $\delta \neq 0$ where $e$ is given by the dotted diagonal in Figure \ref{fig:Ext}(i).  Lemma \ref{lem:type_A_cover}(i) says that $\delta$ is a $( \Sigma^{ -1 }\langle s' \rangle )$-cover, so $L_{ \{ s' \} }( s ) = e$ by Definition \ref{def:tilting}(i).  Hence $L_{ \{ s' \} }( s )$ is given by the dotted diagonal in Figure \ref{fig:L}(i).

For part (ii), assume that $s$ does not have $\dt{s}'-1$ or $\ddt{s}'-1$ as an end point.  Proposition \ref{pro:Ext}(i) says that $\cC( \Sigma^{ -1 }s',s ) = 0$.  There is a trivial triangle $0 \xrightarrow{} s \xrightarrow{ \id_s } s \xrightarrow{} 0$.  Lemma \ref{lem:type_A_cover}(ii) says that $0 \xrightarrow{} s$ is a $( \Sigma^{ -1 }\langle s' \rangle )$-cover, so $L_{ \{ s' \} }( s ) = s$ by Definition \ref{def:tilting}(i), corresponding to Figure \ref{fig:L}(ii).

Part (iii) holds by Definition \ref{def:tilting}(i) and Remark \ref{rmk:CPS2}(iii), corresponding to Figure \ref{fig:L}(iii).
\end{proof}

\begin{Example}
[A simple minded system and its extension closure]
\label{exa:Type_A_Theorems_A_and_B}
Suppose $e=3$ and $w=2$ whence $N=10$ and the polygon $P$ has vertices $\{ 0, 1, \ldots, 9 \}$.  Remark \ref{rmk:CPS2}(i)+(ii) permits to compute the Auslander--Reiten quiver of $\cC$, which is shown in 
Figure \ref{fig:SMS2} where $ij$ is shorthand for the diagonal $\{ i,j \}$.  Remark \ref{rmk:CPS2}(iv) permits writing down the $2$-simple minded system $\cS$ shown in red in Figure \ref{fig:SMS}.

Theorems \ref{thm:A} and \ref{thm:B} say that the extension closure $\langle \cS \rangle$ is a proper abelian subcategory of $\cC$ and that $\langle \cS \rangle \simeq \mod\,A$ for a finite dimensional algebra $A$.  Indeed, in Figure \ref{fig:SMS2} the objects of $\cS$ are red, the remaining indecomposable objects of $\langle \cS \rangle$ are green, and we have $\langle \cS \rangle \simeq \mod\,A$ where $A$ is the algebra $k( \circ \xrightarrow{ \alpha } \circ \xrightarrow{ \beta } \circ )/( \alpha\beta )$.  
\end{Example}

\begin{Example}
[Left tilting of $w$-simple minded systems]
\label{exa:Type_A_Theorem_C}
Suppose $e=2$ and $w=3$ whence $N=10$ and the polygon $P$ has vertices $\{ 0, 1, \ldots, 9 \}$.  Remark \ref{rmk:CPS2}(iv) permits writing down the $3$-simple minded systems in $\cC$; there are $15$ of these shown in the small $10$-gons in Figure \ref{fig:LA}.  Proposition \ref{pro:left_tilting} permits determining the left tilt of each $3$-simple minded system at each of its two indecomposable objects as shown by the arrows in Figure \ref{fig:LA}.
\end{Example}

%
%

\medskip
\noindent
{\bf Acknowledgement.}
The main item of Section \ref{sec:exact} is a theorem by Dyer on exact subcategories of triangulated categories, see \cite[p.\ 1]{Dyer} and Theorem \ref{thm:Dyer}.  We are grateful to Professor Dyer for permitting the proof of Theorem \ref{thm:Dyer} to appear in this paper.

We thank the referee for several useful corrections and suggestions.  We are grateful to Raquel Coelho Sim\~{o}es and David Pauksztello for patiently answering numerous questions and to Osamu Iyama and Haibo Jin for comments on an earlier version.

This work was supported by a DNRF Chair from the Danish National Research Foundation (grant DNRF156), by a Research Project 2 from the Independent Research Fund Denmark (grant 1026-00050B), by the Aarhus University Research Foundation (grant AUFF-F-2020-7-16), and by the Engineering and Physical Sciences Research Council (grant EP/P016014/1).

Thanks to Aleksandr and Sergei the Meerkats for years of inspiration.

\end{document}